\documentclass[11pt,a4paper]{amsart}
\usepackage{amsmath,amsfonts,amsthm,amsopn,color,amssymb,enumitem}
\usepackage[a4paper]{geometry}
\usepackage{palatino}
\usepackage{graphicx}
\usepackage[colorlinks=true]{hyperref}
\hypersetup{urlcolor=blue, citecolor=red, linkcolor=blue}

\usepackage{cite}
\usepackage{relsize}
\usepackage{esint}
\usepackage{verbatim}
\usepackage{mathrsfs}
\usepackage{xcolor}
\usepackage{tikz}

\usepackage{ifpdf}\ifpdf%
\usepackage[utf8]{inputenc}\UseRawInputEncoding{}\else
\usepackage[UTF8]{ctex}\fi
\usepackage[T1]{fontenc}
\usepackage{lmodern}
\usepackage{amsmath,amsthm,amssymb,amsfonts,geometry,mathrsfs}
\usepackage{hyperref,cite}\hypersetup{colorlinks,allcolors=blue,breaklinks=true}
\usepackage[numbers,square,sort&compress]{natbib}
\DeclareMathAlphabet{\mathbbold}{U}{bbold}{m}n
{\theoremstyle{plain}\newtheorem{theorem}{Theorem}[section]\newtheorem{lemma}[theorem]{Lemma}\newtheorem{proposition}[theorem]{Proposition}}\numberwithin{equation}{section}{\theoremstyle{remark}\newtheorem{remark}{\bf Remark}}
\usepackage{xcolor}

\title{Infinitely Many Solutions for a Boundary Yamabe problem}

\author[L. Battaglia]{Luca Battaglia}
\address{\noindent  Dipartimento di Matematica e Fisica, Universit\`a degli Studi Roma Tre}
\email{luca.battaglia@uniroma3.it}

\author[Y. Pu]{Yixing Pu}
\address{\noindent  School of Mathematics and Statistics, East China Normal University, Shanghai, China }
\email{yxpu@stu.ecnu.edu.cn}

\author[G. Vaira]{Giusi Vaira}
\address{\noindent  Dipartimento di Matematica, Universit\`a degli Studi di Bari Aldo Moro, Italy }
\email{giusi.vaira@uniba.it}
\thanks{Work partially supported by 
the MUR-PRIN-P2022YFAJH ``Linear and Nonlinear PDE's: New directions and Applications", by the MUR-PRIN-2022AKNSE ``Variational and Analytic aspects of Geometric PDEs", by the INdAM-GNAMPA project ``Fenomeni non lineari: problemi locali e non locali e loro applicazioni",  CUP E5324001950001 and by China Scholarship Council (No. 202306140126)}

\date{}

\begin{document}

\allowdisplaybreaks{}
\everymath{\displaystyle}

\maketitle

\begin{abstract}
We consider the classical geometric problem of prescribing the scalar and the boundary mean curvature problem in  the unit ball $\mathbb{B}^n$ endowed with the standard Euclidean metric. We will deal with the case of negative scalar curvature showing the existence of infinitely many non-radial positive solutions when $N\geq 5$. This is the first result of existence of solutions in the case of negative prescribed scalar curvature problem in higher dimensions.
\end{abstract}

\section{Introduction}
One of the most important problems in differential geometry is the so-called prescribed curvature problem, i.e.  {\em given $(M,g)$ be a Riemannian closed manifold of dimension $N\ge3$ and a smooth function $\mathbf{K}:M\to\mathbb R$, finding a metric $\tilde g$ conformal to the original metric $g$ whose scalar curvature is $\mathbf{K}$} (see \cite{SchoenYau_1994,ChenLi_2010,Hebey_2000,KazdanWarner_1974}).\\  As it is well known, being $\tilde g=u^{\frac4{N-2}}g$, this is equivalent to finding a positive solution of the semi-linear elliptic equation:
\begin{align}\label{eq.curvature_problem_no_boundary}
-\frac{4(N-1)}{N-2}\Delta_gu+k_gu=\mathbf{K}u^{\frac{N+2}{N-2}},\qquad~ u>0,\qquad&\text{in}~M,
\end{align}
where $k_g$ denotes the scalar curvature of $M$ with respect to $g$ and $\Delta_g$ is the Beltrami-Laplace operator.

If $M$ is a manifold with boundary, given a smooth function $\mathbf{H}:\partial M\to\mathbb R$, it is natural to ask if there exists a conformal metric whose scalar curvature and boundary mean curvature can be prescribed as $\mathbf{K}$ and $\mathbf{H}$ respectively. As in \eqref{eq.curvature_problem_no_boundary}, the geometric problem turns out to be equivalent to a semi-linear elliptic equation with a nonlinear Robin boundary condition:
\begin{align}\label{eq.basicly_original}
\left\{\begin{array}{ll}
-\frac{4(N-1)}{N-2}\Delta_gu+k_gu=\mathbf{K}u^{\frac{N+2}{N-2}},\,\quad u>0,&\text{in }M,\\\\\frac2{N-2}\partial_\nu u+h_gu=\mathbf{H}u^{\frac N{N-2}},&\text{on }\partial M,
\end{array}\right.
\end{align}
where, $k_g$ and $h_g$ denote the scalar and boundary mean curvatures of $M$ with respect to $g$ and $\nu$ is the outward normal unit vector with respect to the metric $g$.

When $\mathbf{K}$ and $\mathbf{H}$ are constants, the problem is known as the Escobar problem, since it was first proposed and studied by Escobar in 1992 in the case $\mathbf{H}=0$ (\cite{Escobar_1992,Escobar_1996}) and in the case $\mathbf{K}=0$ (\cite{Escobar_1992a}). 
Afterwards, many subsequent contributions for this problem are given in \cite{Almaraz, BrendleChen, MayerNdiaye_2017,Marques_2005}.

The case of non-zero constants $\mathbf{K}$ and $\mathbf{H}$ (with $\mathbf{K}>0$) it was first studied by Han \& Li in \cite{HanLi_1999,HanLi_2000} and then completed by Chen, Ruan \& Sun in \cite{ChenRuanSun}.\\ In all these results, the existence of solutions for the problem \eqref{eq.basicly_original} strongly depends on the dimension of the manifold, on the properties of the boundary (i.e. being umbilical or not) and on vanishing properties of the Weyl tensor.\\

The case of non-constant functions $\mathbf{K}$ and $\mathbf{H}$ is less studied and all the available results are for special manifolds (tipically the unit ball and the half-sphere). We refer to \cite{Li_1995,Li_1996,AhmedouBenAyed_2021,BenAyedElMehdiOuldAhmedou_2005,BenAyedElMehdiAhmedou_2002} for the case $\mathbf{H}=0$ and to \cite{AbdelhediChtiouiAhmedou_2008,DjadliMalchiodiAhmedou_2004,XuZhang_2016,ChangXuYang_1998} for  the case $\mathbf{K}=0$.\\

When both $\mathbf{K}$ and $\mathbf{H}$ are not constant, the problem becomes more difficult. Djadli, Malchiodi \& Ahmedou \cite{DjadliMalchiodiOuldAhmedou_2003} considered problem \eqref{eq.basicly_original} on the three-dimensional half-sphere and proved some existence and compactness results. Chen, Ho \& Sun \cite{ChenHoSun_2018} proved the existence of solutions for \eqref{eq.basicly_original} when $\mathbf{K}$ and $\mathbf{H}$ are negative functions and the boundary $\partial M$ has negative Yamabe invariant.
Ambrosetti, Li \& Malchiodi \cite{AmbrosettiLiMalchiodi_2002} considered the perturbation problem in the unit ball when both $\mathbf{K}$ and $\mathbf{H}$ are positive. That is, they consider $\mathbf{K} = \mathbf{K}_0 + \varepsilon \mathcal{K} > 0$ and $\mathbf{H} = \mathbf{H}_0 + \varepsilon \mathcal{H} > 0$, where $\mathbf{K}_0 > 0$, $\mathbf{H}_0 > 0$, $\varepsilon > 0$ is small, and $\mathcal{K}$ and $\mathcal{H}$ are smooth functions. They proved an existence result when $\mathcal{K}$ and $\mathcal{H}$ satisfy some conditions.\\
The first result concerning the case of negative prescribed scalar curvature (namely $\mathbf{K}<0$) is due to Cruz-Bl\'azquez, Malchiodi \& Ruiz in \cite{CruzMalchiodiRuiz}. They introduce the scaling invariant
quantity $$\mathfrak D:=\sqrt{N(N-1)}\frac{H(p)}{\sqrt{|\mathbf K(p)|}},\quad p\in\partial M$$ and established the existence of a solution to \eqref{eq.basicly_original} whenever $\mathfrak D<1$ along the whole boundary. On the other hand, if $\mathfrak D>1$ at some boundary points, they got a solution only in a three dimensional manifold, for a generic choice of $\mathbf K$ and $\mathbf H$.\\ Their proof relies on a careful blow-up analysis: first they show that the blow-up phenomena occur
at boundary points $p$ with $\mathfrak D\geq1$, with different behaviours depending on whether $\mathfrak D=1$ 1or $\mathfrak D>1$. To deal with the loss of compactness at points with $\mathfrak D>1$, where bubbling of solutions occurs, it is shown that in dimension three all the blow-up points are isolated and simple. As a consequence, the number of blow-up points is finite and the blow-up is excluded via integral estimates. In that regard, $N= 3$ is the maximal dimension for which one
can prove that the blow-up points with $\mathfrak D>1$ are isolated and simple for generic choices of $\mathbf K$ and $\mathbf H$. In the case of closed surfaces such a property is assured up to dimension four (see \cite{Li_1996}) but, as observed in \cite{DjadliMalchiodiOuldAhmedou_2003}, the
presence of the boundary produces a stronger interaction of the bubbling solutions with the function $\mathbf K$.\\
A linearly perturbed problem was considered in \cite{Cruz-BlazquezPistoiaVaira_2022} and it is shown that, at least for $4\leq N\leq 7$  the blow-up points are not anymore isolated and simple since a cluster-type solution exists. Moreover in \cite{Cruz-BlazquezVaira_2025} it is addressed the question of existence of solutions for problem \eqref{eq.basicly_original} in a perturbative setting.
Afterwards, in \cite{BattagliaCruz-BlazquezPistoia_2023}, the authors considered the perturbation problem in the ball under the condition $\mathbf{K}<0$ and $\mathbf{H} > 0$. i.e., $\mathbf{K} = \mathbf{K}_0 + \varepsilon \mathcal{K} < 0$ and $\mathbf{H} = \mathbf{H}_0 + \varepsilon \mathcal{H} > 0$, where $\mathbf{K}_0 < 0$, $\mathbf{H}_0 > 0$, $\varepsilon > 0$ is small, and $\mathcal{K}$ and $\mathcal{H}$ are smooth functions showing the existence of solutions with some constraint of $\mathcal{K}$ and $\mathcal{H}$.\\
At this time, as far as we know, it remains completely open whether solutions exist in any kind of manifolds with $\mathbf K<0$ and $\mathbf H$ whatever but non-zero and not constants.\\
In this paper we consider the problem \eqref{eq.basicly_original} in the unit ball, i.e. 
\begin{align}\label{eq.basicly_original_Ball}
    \left\{\begin{array}{ll}
    -\frac{4(N-1)}{N-2}\Delta u=\mathbf{K} u^{\frac{N+2}{N-2}},\,u>0,&\text{in}~\mathbb B^N,\\[10pt]\frac2{N-2}\partial_\nu u+u=\mathbf{H}u^{\frac N{N-2}},&\text{on}~\mathbb S^{N-1}.
    \end{array}\right. 
    \end{align}
and we  focus on the case $\mathbf{K}<0$ and $\mathbf{H} > 0$ are functions which satisfy some conditions showing the existence of infinitely many non-radial solutions to \eqref{eq.basicly_original_Ball}. The solutions we construct depend on a positive integer parameter $k\in\mathbb N$ and, as $k$ tends to $+\infty$, they concentrate on a circle of increasing radius (see Section \ref{sec.notaions} for further details).\\
This type of solutions are known in the case of closed manifold (see \cite{WeiYan_2010a}) and for problem \eqref{eq.basicly_original_Ball} with $\mathbf K=0$ and $\mathbf H>0$  in \cite{WangZhao_2013} (see also \cite{BianChenYang} for some general assumptions on $\mathbf H$ and again $\mathbf K=0$). \\ The basic idea, in order to consider problem \eqref{eq.basicly_original_Ball}, is to use the conformal equivalence between $\mathbb B^N$ and $\mathbb R^N_+$. Indeed, 
\eqref{eq.basicly_original_Ball} is equivalent to the problem:
\begin{align}\label{pb}
  \left\{
    \begin{array}{l}
      -\frac{4(N-1)}{N-2}\Delta u = -K(x) u^{\frac{N+2}{N-2}}, \quad \text{in}~\mathbb{R}_{+}^{N},\vspace{5pt}\\
      -\frac{2}{N-2}\partial_{x_N}u =  H(\bar x) u^{\frac{N}{N-2}}, \quad \text{on}~ \partial\mathbb{R}_{+}^{N},
    \end{array}
  \right.
\end{align}
where $ K:= - \mathbf K\circ\mathscr I >0$ and $H:=\mathbf H\circ\mathscr I >0$ are positive bounded functions in $\mathbb R_+^N$ and $\partial\mathbb R_+^N$ respectively and the conformal map $\mathscr I:\mathbb R^N_+\to\mathbb B^N$ is defined by
\begin{align}\label{eq.inversion_map}
  \mathscr I\left(\bar x,x_N\right)=\left(\frac{2\bar x}{|\bar x|^2+{(x_N+1)}^2},\frac{1-|{\bar x}|^2-x_N^2}{|\bar x|^2+{(x_N+1)}^2}\right),\quad x=(x_1,\ldots, x_N) = \left(\bar x,x_N\right)\in\partial\mathbb R^N_+\times(0,+\infty).\\
  \end{align}
In this paper we consider the simplest situation in which $K$ and $H$ are rotationally invariant, i.e. $ K(x)= K(|x|)$ and $H(\bar{x})= H(|\bar{x}|)$ are positive, bounded and assume that there is a constant $r_0 > 0$ such that 

$$ ({\bf K})  \qquad K(r)=K(r_0)-c_0|r-r_0|^m+\mathcal O(|r-r_0|^{m+\theta}),\quad r\in(r_0-\delta, r_0+\delta),$$
and
$$({\bf H})\qquad 
H(r)=H(r_0)-d_0|r-r_0|^n+\mathcal O(|r-r_0|^{n+\theta}),\quad r\in(r_0-\delta, r_0+\delta),$$
where  $\theta>0$, $\delta>0$  are constants and $m, n\in [2, N-2)$.\\
To make sure that such $m, n$ exist, we consider the problem for $N\geq 5$.\\
We also set
\begin{equation}\label{D}
\mathfrak D := \sqrt{N\left(N-1\right)}\frac{H(r_0)}{\sqrt{K(r_0)}}\end{equation}
and we assume that $\mathfrak D>1$.\\
Without loss of generality, we assume that $K(r_0)=1$. Moreover, by \eqref{D}, we get that
$$H(r_0):=\frac{\mathfrak D}{\sqrt{N(N-1)}}.$$

We finally let $\mathfrak m:=\min\{m, n\}$. Our main result is stated as follows.
\begin{theorem}\label{principale}
Suppose that $N \geq 5$. We let $K(|x|), H(|\bar{x}|)$ satisfy ({\bf K}) and ({\bf H})  respectively, with $m, n \in\left[2,N-2\right)$, and moreover:
$$\left\{\begin{array}{ll}c_0<0,d_0\in\mathbb R,&\mbox{if }m<n,\\c_0<0,d_0>0,&\mbox{if }m=n,\\c_0\in\mathbb R,d_0>0,&\mbox{if }m>n.\end{array}\right.$$
Then problem \eqref{pb} has infinitely many non-radial solutions.
 \end{theorem}
We outline that, going back to the conformal equivalence, Theorem \ref{principale} shows that problem \eqref{eq.basicly_original} in the unit ball with $\mathbf K<0$ has a solution for $N\geq 5$. We remark that this is the first existence result for problem \eqref{eq.basicly_original} in higher dimensions for the pure geometric case (not in a perturbative setting) even if in special cases of manifolds (i.e. in the unit ball).\\ 

Some remarks are in order.
\begin{remark}\label{R1}
The conditions $(\bf{K})$ and $(\bf{ H})$ are local conditions. Moreover, the sign conditions on $c_0$ and $d_0$ in Theorem \ref{principale}  mean that $K$ and $H$ have a local maximum or a local minimum at $r_0$ according to the sign of $c_0$ and $d_0$ respectively.\\ In particular, if $m<n$ then the prevailing phenomena is due to the prescribed scalar curvature and a solution like the one in Theorem \ref{principale} exists if $K$ has a local minimum at $r_0$ (so $-K$ has a local maximum) and $H$ is whatever.
If $m>n$ then we assume that $H$ has a local maximum at $r_0$ and $K$ is whatever and hence the boundary mean curvature determines the existence result. The case $m=n$ is the one in which there is a competition between the two effects and the existence is guaranteed if again $K$ has a local minimum at $r_0$ and $H$ has a local maximum at $r_0$. This last condition can be relaxed (see Remark \ref{ossprimo} in the end).
\end{remark}
\begin{remark}\label{R2}
The radial symmetry of $K$ and $H$ can be relaxed as usual. Indeed, it is, in our opinion, a suitable exercise to let $K:\mathbb R^N_+\to \mathbb R$ bounded and such that for any $(x', x'')\in\mathbb R^2\times\mathbb R^{N-2}$ with $x_N>0$ we have $$K(x', x'')=K(|x'|, x'')=K(r, x''),$$ while $H:\mathbb R^{N-1}\to \mathbb R$ bounded and such that for any $(x', \tilde x'')\in\mathbb R^2\times\mathbb R^{N-3}$ we have $$H(x',\tilde x'')=H(|x'|,\tilde x'')=H(r, \tilde x'').$$ If one assume that $K(r, x'')$ and $H(r, \tilde x'')$ have a common stable critical point at some $(r_0, \tilde x''_0)$ that satisfies suitable assumption then arguing as in \cite{BianChenYang} or as in \cite{PengWangWei} (by using local Pohozaev identities) the result will follow. Here we remark that the pure interest is in the geometrical problem, namely we are interested in looking for examples of existence of solutions for problem \eqref{eq.basicly_original} with $\mathbf{K}<0$ in higher dimensions.\\ We do not known if it is possible to relax over the assumptions. If $K$ and $H$ are not radial in the first two variables but only invariant by a rotation of an angle $\frac{2\pi}{k}$, we do not know how the reduced problem becomes. It would be interesting to investigate this point in the future.\end{remark}

In order to prove Theorem \ref{principale}, we will use a widely used technique in singularly perturbed elliptic problems which takes $k$ which is the number of the bubbles in the solutions as a large parameter. This technique was succefully used also in other contexts and for other problems.\\ Since we construct the solutions by using the so-called bubbles (see \eqref{eq.bubble_form}) then our results provide also an asymptotic behavior of the solutions that we look for (see \eqref{eq:W} and \eqref{eq:W1}).\\\\
The paper is organized as follows. In Section \ref{sec.notaions}, we obtain some preliminary estimates and we introduce the weighted space in which the reduction will be made. In Section \ref{riduzione}, we perform the finite-dimensional reduction studying the nonlinear projected problem. In Section \ref{ridotto}, we come to the variational reduction procedure and we prove Theorem \ref{principale}. Finally we give a list of some useful estimates in Section \ref{appendice}.

\section{Notations and Preliminaries}\label{sec.notaions}
In what follows we set $$c_N:=\frac{4(N-1)}{N-2}.$$ 
Let us fix a large positive integer $k\ge k_0$, for $k_0$ to be determined later, and a scaling parameter $\mu$ defined as $$\mu=k^{\frac{N-2}{N-2-\mathfrak m}}.$$
We remark that $\mu\to +\infty$ as $k\to +\infty$ and in what follows  $\mu$ is a scaling parameter. Indeed, using the transformation $u(y) \mapsto \mu^{\frac{N-2}{2}} u\left(\frac{y}{\mu}\right) $, we find that \eqref{pb} is equivalent to
\begin{align}\label{pbscal}
  \left\{
    \begin{aligned}
      -c_N\Delta u = - K\left(\frac{|y|}{\mu}\right) u^{\frac{N+2}{N-2}}, \quad &\text{in }\mathbb{R}_{+}^{N}, \vspace{5pt}\\
      -\frac{2}{N-2}\partial_{y_N}u= H\left(\frac{|\bar{y}|}{\mu}\right) u^{\frac{N}{N-2}}, \quad &\text{on } \partial\mathbb{R}_{+}^{N}.
    \end{aligned}
  \right.
\end{align}

We recall that for the problem
\begin{align}\label{eq:bubble}
  \left\{\begin{array}{ll}
  -c_N\Delta u=-u^{\frac{N+2}{N-2}},\,u>0,&\text{in }\mathbb R^N_+ ,\vspace{5pt}\\-\frac2{N-2}\partial_{y_N}u=\frac{\mathfrak{D}}{\sqrt{N(N-1)}} u^{\frac N{N-2}},&\text{on }\partial\mathbb R^N_+,\vspace{5pt} \\
  u \in D^{1,2}\left(\mathbb{R}_{+}^{n}\right),&
  \end{array}\right.
  \end{align}
a classification results is known (see \cite{ChipotShafrirFila_1996}) and all solutions of \eqref{eq:bubble} have the form
\begin{align}\label{eq.bubble_form}
  \begin{aligned}
    &U_{z,\Lambda}(\bar y,y_N):=\Lambda^{\frac{N-2}2}U_{\mathbf 0,1}\Bigl(\Lambda \bigl(y-z\bigr)\Bigr)=\frac{\alpha_N\Lambda^{\frac{N-2}{2}}}{{\left({\Lambda^2|\bar{y}-\bar{z}|}^2+{(\Lambda y_N+ \mathfrak{D})}^2 - 1\right)}^{\frac{N-2}{2}}},\\
\end{aligned}
\end{align}
where $\alpha_N:=({4N(N-1)})^{\frac{N-2}{4}}$ and for $\Lambda>0$, $z=(\bar z,0)\in\partial\mathbb R^N_+\times\{0\}$.\\ Recently, in \cite{Cruz-BlazquezPistoiaVaira_2022}, it was shown that $U_{\mathbf 0,1}$ is also non-degenerate, that is if we let $v\in D^{1, 2}(\mathbb R^N_+)$, a solution of the following the linearized problem 
\begin{equation}\label{lblin}
\left\{\begin{aligned}&-c_N\Delta v+\frac{N+2}{N-2}U_{\mathbf 0,1}^{\frac{4}{N-2}}v  =0\quad& \mbox{in}\, \mathbb R^{N}_+\\
&-\frac{2}{N-2}\partial_{y_N}v-\frac{N}{N-2}\frac{\mathfrak{D}}{\sqrt{N(N-1)}}U_{\mathbf 0,1}^{\frac{2}{N-2}}v=0
\quad &\mbox{on}\, \partial\mathbb R^N_+\end{aligned}\right.\end{equation}
 then $v$ is the linear combination of the functions 
\begin{equation}\label{zi}
\mathfrak z^i(y):=\frac{\partial U_{\mathbf 0,1}}{\partial y_i}=\alpha_N\frac{(2-N)y_i}{\left(|\bar y|^2+(y_N+\mathfrak D)^2-1\right)^{\frac N 2}}\quad i=1, \ldots, N-1\end{equation}
and
\begin{equation}\label{zn}\begin{aligned}
\mathfrak z^0(y)&:=\frac{\partial U_{z,\Lambda}}{\partial \Lambda }{\Big|_{(z, \Lambda)=(0, 1 )}}=\left(\frac{2-N}{2} U_{\mathbf 0,1}(y)-\nabla U_{\mathbf 0,1}(y)\cdot (y+\mathfrak D \mathfrak e_n)+\mathfrak D\frac{\partial U_{\mathbf 0,1}}{\partial y_N}\right)\\
&=\alpha_N\frac{N-2}{2}\frac{|y|^2+1-\mathfrak D^2}{\left(|\bar y|^2+(y_N+\mathfrak D)^2-1\right)^{\frac N 2}}.\end{aligned}\end{equation}

We define
$$\begin{aligned}H_s&=\left\{u\in D^{1, 2}(\mathbb R^N_+)\,\,: u\,\,\mbox{is even in}\,\, y_h,\,\,\forall\,\, h=2, \ldots, N-1,\right.\\
&\quad\quad\left. u\left(r \cos \alpha, r \sin  \alpha,y_3,\ldots,y_N\right) = u\left(r \cos \left(\alpha + \frac{2\pi j}{k}\right), r \sin \left(\alpha + \frac{2\pi j}{k}\right),y_3,\ldots,y_N\right)\right\}.\end{aligned}$$
We also let \begin{align*}
  \mathbf{x}_j := \left(\bar{\mathbf{x}}_j,0\right)= \left(r \cos \frac{2\left(j-1\right)\pi}{k}, r \sin \frac{2\left(j-1\right)\pi}{k}, 0, \ldots, 0\right) \in \mathbb{R}^{2}\times \mathbb R^{N-2}, \quad j=1, \ldots, k,
\end{align*}
\begin{align*}
  r \in \left[\mu r_0 - \frac{1}{\mu^{\bar{\theta}}},\mu r_0 + \frac{1}{\mu^{\bar{\theta}}}\right],\quad L_0 \leq \Lambda \leq L_1,
\end{align*}
and $\bar{\theta}>0$ is a small number, $L_1>L_0>0$ are some constants. \\ We remark that $\Lambda$ is independent of $k$, while $r$ depends on $\mu$ (and hence on $k$). This is why the solutions we want to build concentrate around circles on a radius that increases as $k$ increases.\\ 
Hence, we define
\begin{align}\label{eq:W}
  W_{r,\Lambda}(y) := \sum_{j=1}^{k}U_{\mathbf{x}_j,\Lambda}.
\end{align}
We will find the solution to \eqref{pbscal} in the form \begin{equation}\label{eq:W1}W_{r,\Lambda} + \phi,\end{equation} with $\phi\in H_s$ satisfying a suitable set of orthogonality conditions and having small norm in some space that will be better specified later.\\ We finally observe that
\begin{align}\label{eq.observation_trans}
  \frac1{|\bar y - \bar{z}|^2+\left|y_n+\mathfrak D\right|^2-1} \leq \frac1{\left|y - \left(\bar{z}, 0\right)\right|^2+\mathfrak{D}^2-1} \leq C\frac{1}{(\left|y - \left(\bar{z}, 0\right)\right| + 1)^2},
\end{align}
which means
\begin{align*}
  U_{\mathbf{x}_i,\Lambda}(y)\leq C\frac{1}{{(\left|y - \left(\bar{\mathbf{x}}_i, 0\right)\right| + 1)}^{N-2}}.
\end{align*}
At the end we also define
$$\Omega_\ell:=\left\{ y=(y',y_3\ldots, y_N)\in \mathbb R^2\times \mathbb R^{N-3}\times \mathbb R_+\,\,: \left\langle \frac{y'}{|y'|}, \frac{ \mathbf{x}_\ell}{|\mathbf{x}_\ell|}\right\rangle\geq \cos\frac\pi k\right\}$$ and 
$$\widetilde\Omega_\ell:=\Omega_\ell\cap\partial\mathbb R^N_+=\left\{ y=(y',y_3\ldots,y_{N-1}, 0)\in \mathbb R^2\times \mathbb R^{N-3}\times \{0\}\,\,: \left\langle \frac{y'}{|y'|}, \frac{ \mathbf{x}_\ell}{|\mathbf{x}_\ell|}\right\rangle\geq \cos\frac\pi k\right\}.$$
 Throughout this paper, we denote $C$ as a generic positive constant independent of $k$. Moreover, when we write $f\lesssim g$ we intend that there exists some positive $C>0$ independent of $k$ so that $f\leq C g$.
\section{Finite-dimensional reduction}\label{riduzione}
In this section we perform a finite-dimensional reduction. To do so we will introduce the following weighted norms
\begin{equation}\label{norm1}
\|\phi\|_{*}=\sup _{y \in \mathbb{R}_{+}^{N}}\left(\sum_{j=1}^{k} \frac{1}{\left(1+\left|y-\mathbf{x}_{j}\right|\right)^{\frac{N-2}{2}+\tau}}\right)^{-1}|\phi(y)|,
\end{equation}
\begin{equation}\label{norm2}
\|h_1\|_{**}=\sup _{y \in \mathbb{R}_{+}^{N}}\left(\sum_{j=1}^{k} \frac{1}{\left(1+\left|y-\mathbf{x}_{j}\right|\right)^{\frac{N+2}{2}+\tau}}\right)^{-1}|h_1(y)|,
\end{equation}
and
\begin{equation}\label{norm3}
  \begin{aligned}
    \|h_2\|_{***}&=\sup _{y \in \partial\mathbb{R}_{+}^{N}}\left(\sum_{j=1}^{k} \frac{1}{\left(1+\left|y-\mathbf{x}_{j}\right|\right)^{\frac{N}{2}+\tau}}\right)^{-1}|h_2(y)|,\\
    &=\sup _{\bar{y} \in \partial\mathbb R^N_+}\left(\sum_{j=1}^{k} \frac{1}{\left(1+\left|\bar{y}-\bar{\mathbf{x}}_{j}\right|\right)^{\frac{N}{2}+\tau}}\right)^{-1}|h_2(\bar{y},0)|,
  \end{aligned}
  \end{equation}
where $\tau>0$ is fixed and small enough. 
\vspace{5pt}
Define \begin{equation}\label{ker}
Z_{i, 1}:=\frac{\partial U_{\mathbf{x}_i,\Lambda}}{\partial r},\qquad Z_{i, 2}:=\frac{\partial U_{\mathbf{x}_i,\Lambda}}{\partial \Lambda},\quad i=1, \ldots, k.\end{equation}
It is easy to show that
\begin{align}\label{eq.Z_i_1_2_expansion_1}
  &Z_{i,1} := \frac{\partial U_{\mathbf{x}_i,\Lambda}}{\partial r} = U_{\mathbf{x}_i,\Lambda}\frac{(N-2)\Lambda^2 \left(y-\mathbf{x}_i\right)}{{\Lambda^2|\bar{y}-\bar{\mathbf{x}}_{i}|}^2+{(\Lambda y_N+ \mathfrak{D})}^2-1} \cdot \frac{\mathbf{x}_i}{r},\\\nonumber\\ \label{eq.Z_i_1_2_expansion_2}
  &Z_{i,2} := \frac{\partial U_{\mathbf{x}_i,\Lambda}}{\partial \Lambda} =U_{\mathbf{x}_i,\Lambda} \frac{N-2}{2\Lambda} \cdot \frac{\mathfrak{D}^2-\Lambda^2 y_N^2-\Lambda^2\left|\bar{y}-\bar{\mathbf{x}}_i\right|^2-1}{{\Lambda^2|\bar{y}-\bar{\mathbf{x}}_{i}|}^2+{(\Lambda y_N+ \mathfrak{D})}^2-1}.
\end{align}
The strategy to solve \eqref{pbscal} is the following: since we look for a solution of the form $W_{r,\Lambda} + \phi$, we first solve a nonlinear projected problem, i.e. we solve the problem
\begin{align}\label{eq.intermediate_nonlinear_form}
  \left\{\begin{array}{l}
 \mathscr L_{in}(\phi)    = \mathscr E_{in}+\mathscr N_{in}(\phi) + \sum_{\ell=1}^{2} \mathfrak c_{\ell}\sum_{i=1}^{k} U_{\mathbf{x}_i,\Lambda}^{\frac{4}{N-2}}Z_{i,\ell}, \quad \text{in}~\mathbb{R}_{+}^{N},\vspace{5pt}\\
    \mathscr L_{bd}(\phi) = \mathscr E_{bd}+\mathscr N_{bd}(\phi) + \sum_{\ell=1}^{2} \mathfrak c_{\ell}\sum_{i=1}^{k} U_{\mathbf{x}_i,\Lambda}^{\frac{2}{N-2}}Z_{i,\ell}, \quad \text{on}~ \partial\mathbb{R}_{+}^{N},\vspace{5pt}\\
    \phi \in H_s,\vspace{5pt}\\
    \langle Z_{i, \ell}, \phi\rangle=0 \quad \text{for}~ i=1,\ldots,k, ~ \ell = 1,2,
  \end{array}\right.
\end{align}
for some numbers $\mathfrak c_{\ell}$, where 
\begin{align*}
 \langle Z_{i, \ell}, \phi\rangle:=  -2\sqrt{N(N-1)}\mathfrak D\int_{\partial\mathbb R^N_+}U_{\mathbf{x}_i,\Lambda}^{\frac{2}{N-2}}Z_{i,\ell}\phi +(N+2)\int_{\mathbb R^N_+} U_{\mathbf{x}_i,\Lambda}^{\frac{4}{N-2}}Z_{i,\ell}\phi \end{align*}
and the linear operators $\mathscr L_{in}, \mathscr L_{bd}$ are defined by
$$ \mathscr L_{in}(\phi):=-c_N\Delta \phi +\frac{N+2}{N-2} K\left(\frac{|y|}{\mu}\right) W_{r,\Lambda}^{\frac{4}{N-2}}\phi,$$
$$\mathscr L_{bd}(\phi):=-\frac{2}{N-2}\partial_{y_N}\phi - \frac{N}{N-2} H\left(\frac{|\bar{y}|}{\mu}\right)  W_{r,\Lambda}^{\frac{2}{N-2}}\phi,$$
while the error terms $\mathscr E_{in},\mathscr E_{bd}$ and 
the nonlinear terms $\mathscr N_{in},\mathscr N_{bd}$ are defined by
\begin{align*}
  \mathscr E_{in}&:=K\left(\frac{|y|}{\mu}\right) W_{r,\Lambda}^{\frac{N+2}{N-2}}+c_N \Delta W_{r,\Lambda},\\
  \mathscr N_{in}(\phi)&:=K\left(\frac{|y|}{\mu}\right)\left[{\left(W_{r,\Lambda} + \phi\right)}^{\frac{N+2}{N-2}} - W_{r,\Lambda}^{\frac{N+2}{N-2}} - \frac{N+2}{N-2} W_{r,\Lambda}^{\frac{4}{N-2}}\phi\right],\\
  \mathscr E_{bd}&:=H\left(\frac{|\bar{y}|}{\mu}\right) W_{r,\Lambda}^{\frac{N}{N-2}}+\frac{2}{N-2} \partial_{y_N} W_{r,\Lambda},\\
  \mathscr N_{bd}(\phi)&:=H\left(\frac{|\bar{y}|}{\mu}\right)\left[{\left(W_{r,\Lambda} + \phi\right)}^{\frac{N}{N-2}} - W_{r,\Lambda}^{\frac{N}{N-2}} - \frac{N}{N-2} W_{r,\Lambda}^{\frac{2}{N-2}}\phi\right].
\end{align*}
After that we will find the parameters $\Lambda$ and $r$ so that $\mathfrak c_\ell=0$ for $\ell=1, 2$.\\
\subsection{Estimate of the error terms}
Here we estimate $\|\mathscr E_{in}\|_{**}$ and $\|\mathscr E_{bd}\|_{***}$ in terms of $\mu$.
\begin{lemma}\label{lemma.error_term_estimates}
  If $N\geq5$, then there exists $\sigma >0$ such that for $k$ sufficiently large it follows that
  \begin{align*}
    \|\mathscr E_{in}\|_{**} \leq C\left(\frac{1}{\mu}\right)^{\frac{m}{2}+\sigma}, \quad  \|\mathscr E_{bd}\|_{***} \leq C\left(\frac{1}{\mu}\right)^{\frac{n}{2}+\sigma}.
  \end{align*}
\end{lemma}

\begin{proof}
The proof is similar to Lemma 2.5 of \cite{WeiYan_2010a} and Lemma 2.1 of \cite{WangZhao_2013}, so we omit the details.
\end{proof}

\subsection{The linear operators}
In order to study the invertibility of the linear part, we will consider the following linear problem
\begin{align}\label{eq.intermediate_linearized_form}
  \left\{\begin{array}{l}
 \mathscr L_{in}(\phi)    = h_{in} + \sum_{\ell=1}^{2} \mathfrak c_{\ell}\sum_{i=1}^{k} U_{\mathbf{x}_i,\Lambda}^{\frac{4}{N-2}}Z_{i,\ell}, \quad \text{in}~\mathbb{R}_{+}^{N},\vspace{5pt}\\
    \mathscr L_{bd}(\phi) = h_{bd} + \sum_{\ell=1}^{2} \mathfrak c_{\ell}\sum_{i=1}^{k} U_{\mathbf{x}_i,\Lambda}^{\frac{2}{N-2}}Z_{i,\ell}, \quad \text{on}~ \partial\mathbb{R}_{+}^{N},\vspace{5pt}\\
    \phi \in H_s,\vspace{5pt}\\
    {\left\langle Z_{i,\ell},\phi \right\rangle} = 0,  \quad \text{for}~ i=1,\ldots,k, ~ \ell = 1,2,
  \end{array}\right.
\end{align}
For any fixed $y = \left(y_1,\ldots, y_N \right) \in \mathbb{R}_{+}^{N} $, we denote by $G(x,y)$ the Green's function of the problem
\begin{align}
  \left\{\begin{array}{ll}  
  -\Delta G(x,y) = \delta_y,&\text{for } x \in \mathbb{R}_{+}^{N},\vspace{5pt}\\
  G(x,y) = 0 , &\text{for } \left|x\right| \rightarrow \infty,\vspace{5pt}\\
  \partial_{x_N}G (x,y)=0,&\text{for } x_N = 0.
  \end{array}\right.
\end{align}
We remark that $G(x,y)$ has the explicit expression:
\begin{align}\label{eq.Green_function_expression}
  G(x,y) = \frac{1}{\omega_N \left(N-2\right)}\left(\frac{1}{{\left|x-y\right|}^{N-2}} + \frac{1}{{\left|x-y^s\right|}^{N-2}}\right),
\end{align}
where $\omega_N$ is the volume of the unit ball in $\mathbb{R}^N$ and $y^s=\left(\bar{y},-y_N\right)$ which means the symmetric point of $y$ with respect to $\partial\mathbb{R}_{+}^{N}$.

\begin{lemma}\label{eq.constraint_0_norm}
  Assume that $\phi_k$ solves \eqref{eq.intermediate_linearized_form} for $h=(h_{in},h_{bd}) =\left((h_{in})_k, (h_{bd})_k\right)=h_k$.\\ If $$\left\Vert h_{in}\right\Vert_{**}\underset{k\to+\infty}\to0,\qquad\left\Vert h_{bd}\right\Vert_{***}\underset{k\to+\infty}\to0,$$
then
$$\left\Vert \phi_k\right\Vert_*\underset{k\to+\infty}\to0.$$
\end{lemma}
\begin{proof}
We argue by contradiction.

  Assume that there are $k \rightarrow+\infty$, $h=h_{k}$, $\Lambda_{k} \in\left[L_{1}, L_{2}\right]$, $r_{k} \in\left[r_{0} \mu-\frac{1}{\mu^{\bar{\theta}}}, r_{0} \mu+\frac{1}{\mu^{\bar{\theta}}}\right]$, and $\phi_{k} \in D^{1,2}\left(\mathbb{R}_{+}^{n}\right)$ solving \eqref{eq.intermediate_linearized_form} for $h=h_{k}, \Lambda=\Lambda_{k}, r=r_{k}$, with $\left\|(h_{in})_{k}\right\|_{* *} \rightarrow 0$, $\left\|(h_{bd})_{k}\right\|_{***} \rightarrow 0$, and $\left\|\phi_{k}\right\|_{*} \geq c'$ for some constant $c'>0$. We may assume, without loss of generality, that $\left\|\phi_{k}\right\|_{*}=1$. For simplicity, we drop the subscript $k$.

First, we estimate $\mathfrak c_{\ell}$, $\ell =1,2$. Multiplying \eqref{eq.intermediate_linearized_form} by $Z_{1, \ell}$, integrating and using the equation satisfied by $Z_{1,\ell}$ we find that

\begin{equation}\label{eq.matrix_cj}
  \begin{aligned}
 & \sum_{j=1}^2\mathfrak c_j\sum_{i=1}^k\int_{\mathbb R^{N}_+}U_{\mathbf{x}_i,\Lambda}^{\frac{4}{N-2}}Z_{i,\ell}Z_{1, \ell}+C_N\frac{N-2}{2}\sum_{j=1}^2\mathfrak c_j\sum_{i=1}^k\int_{\partial\mathbb R^N_+}U_{\mathbf{x}_i,\Lambda}^{\frac{2}{N-2}}Z_{i,\ell}Z_{1, \ell}\\
 =&-C_N\frac{N-2}{2}\int_{\partial\mathbb R^N_+}h_{bd}Z_{1, \ell}-\int_{\mathbb R^{N}_+}h_{in}Z_{1, \ell}+\frac{N+2}{N-2}\int_{\mathbb R^N_+}\left(K\left(\frac{|y|}{\mu}\right)-1\right)U_{\mathbf{x}_1,\Lambda}^{\frac{4}{N-2}}Z_{1, \ell}\phi\\&+\frac{N+2}{N-2}\int_{\mathbb R^N_+}K\left(\frac{|y|}{\mu}\right)\left(W_{r, \Lambda}^{\frac{4}{N-2}}-U_{\mathbf{x}_1,\Lambda}^{\frac{4}{N-2}}\right)Z_{1, \ell}\phi-\frac{c_N N}{2}\int_{\partial\mathbb R^N_+}H\left(\frac{|\bar y|}{\mu}\right)\left(W_{r, \Lambda}^{\frac{2}{N-2}}-U_{\mathbf{x}_1,\Lambda}^{\frac{2}{N-2}}\right)Z_{1, \ell}\phi\\&-\frac{c_N N}{2}\int_{\partial\mathbb R^N_+}\left(H\left(\frac{|\bar y|}{\mu}\right)-\frac{\mathfrak D}{\sqrt{N(N-1)}}\right)U_{\mathbf{x}_1,\Lambda}^{\frac{2}{N-2}}Z_{1, \ell}\phi.
  \end{aligned}
\end{equation}
In a standard way it follows that
$$\sum_{j=1}^2\mathfrak c_j\sum_{i=1}^k\int_{\mathbb R^{N}_+}U_{\mathbf{x}_i,\Lambda}^{\frac{4}{N-2}}Z_{i,\ell}Z_{1, \ell}+C_N\frac{N-2}{2}\sum_{j=1}^2\mathfrak c_j\sum_{i=1}^k\int_{\partial\mathbb R^N_+}U_{\mathbf{x}_i,\Lambda}^{\frac{2}{N-2}}Z_{i,\ell}Z_{1, \ell}=\mathfrak c_\ell(A+o(1)),$$
for some $A>0$.\\
Moreover, it follows from Lemma \ref{lemma.appendix_estimates_1} and by \eqref{eq.lemma.appendix_estimates_3_temp}, \eqref{eq.Z_i_1_2_expansion_1} and \eqref{eq.Z_i_1_2_expansion_2} that,
for any \\ $1<\alpha\leq \min\left\{N-2, \frac{N+2}{2}+\tau\right\}$,
\begin{align*}
  \begin{aligned}
  \int_{\mathbb R^N_+} h_{in}Z_{1, \ell}& \lesssim \|h_{in}\|_{**} \int_{\mathbb{R}_{+}^{N}} \frac{1}{\left(1+\left|y-\mathbf{x}_{1}\right|\right)^{N-2}} \sum_{i=1}^{k} \frac{1}{\left(1+\left|y-\mathbf{x}_{i}\right|\right)^{\frac{N+2}{2}+\tau}} d y \\
  & \lesssim\|h_{in}\|_{**}\left(1+\sum_{j=2}^k\frac{1}{|\mathbf x_j-\mathbf x_1|^\alpha}\int_{\mathbb R^N_+}\frac{1}{\left(1+\left|y-\mathbf{x}_{j}\right|\right)^{\frac{3}{2}N-1+\tau-\alpha}}\right)\\&\lesssim\|h_{in}\|_{**} .    \end{aligned}
\end{align*}

Similarly, we have
$$\int_{\partial\mathbb R^N_+}h_{bd}Z_{1, \ell}\lesssim \|h_{bd}\|_{***}. $$

Now, by arguing as in \cite{WeiYan_2010a}
it follows that
$$\int_{\mathbb R^N_+}\left(K\left(\frac{|y|}{\mu}\right)-1\right)U_{\mathbf{x}_1,\Lambda}^{\frac{4}{N-2}}Z_{1, \ell}\phi\lesssim o(1)\|\phi\|_*, \qquad \int_{\mathbb R^N_+}K\left(\frac{|y|}{\mu}\right)\left(W_{r, \Lambda}^{\frac{4}{N-2}}-U_{\mathbf{x}_1,\Lambda}^{\frac{4}{N-2}}\right)Z_{1, \ell}\phi\lesssim o(1)\|\phi\|_*$$

and by arguing as in \cite{WangZhao_2013} we get that
$$\int_{\partial\mathbb R^N_+}H\left(\frac{|\bar y|}{\mu}\right)\left(W_{r, \Lambda}^{\frac{2}{N-2}}-U_{\mathbf{x}_1,\Lambda}^{\frac{2}{N-2}}\right)Z_{1, \ell}\phi\lesssim o(1)\|\phi\|_*$$ and $$ \int_{\partial\mathbb R^N_+}\left(H\left(\frac{|\bar y|}{\mu}\right)-\frac{\mathfrak D}{\sqrt{N(N-1)}}\right)U_{\mathbf{x}_1,\Lambda}^{\frac{2}{N-2}}Z_{1, \ell}\phi\lesssim o(1)\|\phi\|_*.$$

We obtain 
\begin{equation}\label{eq.c_ell_estimate}
|c_{\ell}|\lesssim \left(\|h_{in}\|_{**}+\|h_{bd}\|_{***}\right)+o(1)\|\phi\|_{*} = o(1).
\end{equation}

By Green's representation formula, we can rewrite \eqref{eq.intermediate_linearized_form} as
  \begin{align*}
    \begin{aligned}
      \phi(y) ={}& \frac{1}{c_N}\int_{\mathbb{R}_{+}^{N}} G(y,z) \left[ -\frac{N+2}{N-2} K\left(\frac{|z|}{\mu}\right) W_{r,\Lambda}^{\frac{4}{N-2}}(z)\phi(z) + h_{in}(z) + \sum_{\ell=1}^{2} \mathfrak c_{\ell}\sum_{i=1}^{k} U_{\mathbf{x}_i,\Lambda}^{\frac{4}{N-2}}(z) Z_{i,\ell}(z) \right] dz \\
      &+ \frac{N-2}{2}\int_{\partial\mathbb{R}_{+}^{N}} G(y,z)\left[ \frac{N}{N-2} H\left(\frac{|\bar{z}|}{\mu}\right)  W_{r,\Lambda}^{\frac{2}{N-2}}(z)\phi(z) + h_{bd}(z) +\sum_{\ell=1}^{2} \mathfrak c_{\ell}\sum_{i=1}^{k} U_{\mathbf{x}_i,\Lambda}^{\frac{2}{N-2}}(z)Z_{i,\ell}(z) \right] d\bar{z},
    \end{aligned}
  \end{align*}
where $G$ is given as in \eqref{eq.Green_function_expression}.
Since $|G(y,z)| \leq \frac{C}{|z-y|^{N-2}},$ we have
\begin{align*}
  \begin{aligned}
    |\phi(y)| &\leq  C\|\phi\|_{*} \int_{\mathbb{R}_{+}^{N}} \frac{1}{|z-y|^{N-2}} W_{r, \Lambda}^{\frac{4}{N-2}} \sum_{j=1}^{k} \frac{1}{\left(1+\left|z-\mathbf{x}_{j}\right|\right)^{\frac{N-2}{2}+\tau}} d z \\
    & \hspace{1em} +C\|h_{in}\|_{* *} \int_{\mathbb{R}_{+}^{N}} \frac{1}{|z-y|^{N-2}} \sum_{j=1}^{k} \frac{1}{\left(1+\left|z-\mathbf{x}_{j}\right|\right)^{\frac{N+2}{2}+\tau}} d z\\
    & \hspace{1em} + C\sum_{\ell=1}^{2} |\mathfrak c_{\ell}| \sum_{i=1}^{k} \int_{\mathbb{R}_{+}^{N}} \frac{1}{|z-y|^{N-2}} \frac{1}{\left(1+\left|z-\mathbf{x}_{i}\right|\right)^{N+2}} d z \\
    & \hspace{1em} + C\|\phi\|_{*} \int_{\partial\mathbb{R}_{+}^{N}} \frac{1}{|z-y|^{N-2}} W_{r, \Lambda}^{\frac{2}{N-2}} \sum_{j=1}^{k} \frac{1}{\left(1+\left|z-\mathbf{x}_{j}\right|\right)^{\frac{N-2}{2}+\tau}} d z \\
    & \hspace{1em} +C\|h_{bd}\|_{***} \int_{\partial\mathbb{R}_{+}^{N}} \frac{1}{|z-y|^{N-2}} \sum_{j=1}^{k} \frac{1}{\left(1+\left|z-\mathbf{x}_{j}\right|\right)^{\frac{N}{2}+\tau}} d z\\
    & \hspace{1em} + C\sum_{\ell=1}^{2} |\mathfrak c_{\ell} |\sum_{i=1}^{k} \int_{\partial\mathbb{R}_{+}^{N}} \frac{1}{|z-y|^{N-2}} \frac{1}{\left(1+\left|z-\mathbf{x}_{i}\right|\right)^{N}} d z\\
  \end{aligned}
\end{align*}
Now, by applying Lemma \ref{lemma.appendix_estimates_3} we have that
\begin{align*}
  \begin{aligned}
    \int_{\mathbb{R}_{+}^{N}} \frac{1}{|z-y|^{N-2}} W_{r, \Lambda}^{\frac{4}{N-2}} \sum_{j=1}^{k} \frac{1}{\left(1+\left|z-\mathbf{x}_{j}\right|\right)^{\frac{N-2}{2}+\tau}} d z & \leq C \sum_{j=1}^{k} \frac{1}{\left(1+\left|y-\mathbf{x}_{j}\right|\right)^{\frac{N-2}{2}+\tau+\theta}}.
   \end{aligned}
\end{align*}
Moreover, by using Lemma \ref{lemma.appendix_estimates_2} with $\sigma=\frac{N-2}{2}+\tau\in (0, N-2)$ by choosing $\tau<\frac{N-2}{2}$. Then  we get 
\begin{align*}
  \begin{aligned}
     \int_{\mathbb{R}_{+}^{N}} \frac{1}{|z-y|^{N-2}} \sum_{j=1}^{k} \frac{1}{\left(1+\left|z-\mathbf{x}_{j}\right|\right)^{\frac{N+2}{2}+\tau}} d z    & \leq \sum_{j=1}^{k} \frac{1}{\left(1+\left|y-\mathbf{x}_{j}\right|\right)^{\frac{N-2}{2}+\tau}},
  \end{aligned}
\end{align*}
and again, by using Lemma \ref{lemma.appendix_estimates_2} with $\sigma=\frac{N-2}{2}+\tau\in (0, N-2)$ by choosing $\tau<\frac{N-2}{2}$. Then  we get

\begin{align*}
  \begin{aligned}
    \sum_{i=1}^{k} \int_{\mathbb{R}_{+}^{N}} \frac{1}{|z-y|^{N-2}} \frac{1}{\left(1+\left|z-\mathbf{x}_{i}\right|\right)^{N+2}} d z &\leq   \sum_{i=1}^{k} \int_{\mathbb{R}_{+}^{N}} \frac{1}{|z-y|^{N-2}} \frac{1}{\left(1+\left|z-\mathbf{x}_{i}\right|\right)^{\frac{N+2}{2}+\tau}} d z\\
    &\leq C \sum_{j=1}^{k} \frac{1}{\left(1+\left|y-\mathbf{x}_{j}\right|\right)^{\frac{N-2}{2}+\tau}}.
  \end{aligned}
\end{align*}
The boundary terms can be treated in a similar way. Hence
\begin{align*}
  \begin{aligned}
    |\phi(y)| 
    & \leq C \|\phi\|_{*}\sum_{j=1}^{k} \frac{1}{\left(1+\left|y-\mathbf{x}_{j}\right|\right)^{\frac{N-2}{2}+\tau+\theta}}+ C (\|h_{in}\|_{**}+\|h_{bd}\|_{***}) \sum_{j=1}^{k} \frac{1}{\left(1+\left|y-\mathbf{x}_{j}\right|\right)^{\frac{N-2}{2}+\tau}},
  \end{aligned}
\end{align*}
for some positive constant $\theta>0$.
Thus, from the definitions \eqref{norm1}, \eqref{norm2}, \eqref{norm3},
\begin{equation}\label{212}
  \|\phi\|_{*} \lesssim  (\|h_{in}\|_{**}+\|h_{bd}\|_{***}) + \displaystyle\frac{\displaystyle\sum_{j=1}^{k}\frac{1}{\left(1+\left|y-\mathbf{x}_{j}\right|\right)^{\frac{N-2}{2}+\tau+\theta}}}{ \displaystyle\sum_{j=1}^{k} \frac{1}{\left(1+\left|y-\mathbf{x}_{j}\right|\right)^{\frac{N-2}{2}+\tau}}}.
\end{equation}
Since $\|\phi\|_{*}=1$, from \eqref{212} it follows that there exists $R>0$ such that
\begin{align}\label{infty}
  \|\phi\|_{L^{\infty}\left(B_{R}(\mathbf{x}_i)\right)} \geq a>0,
\end{align}
for some $i$.

But, since $\|\phi\|_*=1$, by elliptic regularity, we have that $\|\phi\|_{C^1(\mathbb R^N_+)}\leq 1$ and hence, by applying the Ascoli-Arzela's Theorem, we get the existence of a subsequence such that $\phi(y-\mathbf{x}_i)$ converges uniformly in compact sets to some $\hat\phi$. As usual, $\hat\phi$ satisfies the limit problem 
\begin{align}
  \left\{\begin{array}{l}
    -c_N\Delta \hat{\phi} +\frac{N+2}{N-2}  U_{\mathbf{0},\Lambda}^{\frac{4}{N-2}} \hat{\phi} = 0, \quad \text{in}~\mathbb{R}_{+}^{N},\vspace{5pt}\\
    -\frac{2}{N-2}\partial_{y_N}\hat{\phi} - \frac{N}{N-2} U_{\mathbf{0},\Lambda}^{\frac{2}{N-2}} \hat{\phi} = 0, \quad \text{on}~ \partial\mathbb{R}_{+}^{N},
  \end{array}\right.
\end{align}
for some $\Lambda \in\left[L_{1}, L_{2}\right]$. Moreover it follows also that (passing to the limit into the orthogonality condition) $\langle \hat \phi, \mathfrak z^i\rangle=0$, so the non-degeneracy of $ U_{\mathbf{0},\Lambda}$ \cite[Theorem 2.1]{Cruz-BlazquezPistoiaVaira_2022} implies that $\hat{\phi}=0$ and it is a contradiction with \eqref{infty}.
\end{proof}
From Lemma \ref{eq.constraint_0_norm}, using the same argument as in the proof of \cite[Proposition 4.1]{delPinoFelmerMusso_2004}, we can prove the following result.

\begin{proposition}\label{prop.linearized_problem_temp_result}
  There exists $k_{0}\in\mathbb N$ and a constant $C>0$, independent of $k$, such that for all $k \geq k_{0}$ and all $\left(h_{in},h_{bd}\right) \in L^{\infty}\left(\mathbb{R}_{+}^{N}\right) \times L^{\infty}\left(\partial\mathbb R^N_+\right)$, problem \eqref{eq.intermediate_linearized_form} has a unique solution $\phi$ denoted as $\phi := L_{k}\left(h_{in},h_{bd}\right)$. Besides,

\begin{equation}
\left\| L_{k}\left(h_{in},h_{bd}\right)\right\|_{*} \leq C\|h_{in}\|_{* *} + C \|h_{bd}\|_{***}, \quad\left|\mathfrak c_{\ell}\right| \leq C\|h_{in}\|_{* *} +C \|h_{bd}\|_{***}, ~\text{for}~ \ell =1,2.
\end{equation}
\end{proposition}
\subsection{The fixed point argument}
As usual, by using a standard contraction argument we can show the existence of $\phi$ solution of \eqref{eq.intermediate_nonlinear_form} (for the proof, see Proposition 2.3 in \cite{WeiYan_2010a}).
\begin{proposition}\label{prop.intermediate_nonlinear_problem}
  There exists $k_0 > 0$, such that for each $k \geq k_0$, $L_0 \leq \Lambda \leq L_1$, $\left|r-\mu r_0\right| \leq \frac{1}{\mu^{\bar{\theta}}}$
 where $\bar{\theta}$ is a fixed small constant, \eqref{eq.intermediate_nonlinear_form} has a unique solution $\phi = \phi (r,\Lambda)$, satisfying
\begin{align*}
  \|\phi\|_{*} \leq C\left(\frac{1}{\mu}\right)^{\frac{\mathfrak m}{2}+\sigma}, \quad\left|\mathfrak c_{\ell}\right| \leq C\left(\frac{1}{\mu}\right)^{\frac{\mathfrak m}{2}+\sigma}, \quad \ell=1,2,
 \end{align*}

 where $\sigma>0$ is a small constant.
\end{proposition}

\section{The reduced problem: proof of Theorem \ref{principale}}\label{ridotto}
In this section we need to find $\Lambda$ and $r$ so that $\mathfrak c_\ell$ is zero for $\ell=1, 2$. \\
The problem \eqref{pb} is variational, that is its solutions are the critical points of the $C^2-$ functional $$J(u):=\frac{c_N}{2}\int_{\mathbb R^N_+}|\nabla u|^2+\frac{1}{2^*}\int_{\mathbb R^N_+} K\left(\frac{|y|}{\mu}\right)(u)^{2^*}_+-(N-2)\int_{\partial \mathbb R^N_+}H\left(\frac{|\bar y|}{\mu}\right)(u)^{2^\sharp}_+$$ where $2^*:=\frac{2N}{N-2}$ and $2^\sharp:=\frac{2(N-1)}{N-2}$.\\
We also define the reduced functional
$$F(r, \Lambda):=J(W_{r, \Lambda}+\phi)$$ where $\phi$ is the function found in Proposition \ref{prop.intermediate_nonlinear_problem}.\\\\
First we need a condition for which $\mathfrak c_\ell=0$.
\begin{proposition}
If $(r, \Lambda)$ is a critical point of the reduced functional $F(r, \Lambda)$ then $\mathfrak c_1=\mathfrak c_2=0$.
\end{proposition}
\begin{proof}
Since $\frac{\partial F}{\partial r}= \frac{\partial F}{\partial \Lambda}=0 \vspace{5pt}$, we only need to show the non-degeneracy of the coefficient matrix with respect to $(\mathfrak c_1,\mathfrak c_2)$ in the following system.
  \begin{align*}
    \begin{pmatrix}
      M_{1,1} & M_{1,2}\\
      M_{2,1} & M_{2,2}
    \end{pmatrix}
    \begin{pmatrix}
    \mathfrak c_1 \vspace{5pt}\\ \mathfrak c_2
    \end{pmatrix} =
    \begin{pmatrix}
      \frac{\partial F}{\partial r} \vspace{10pt}\\
      \frac{\partial F}{\partial \Lambda}
    \end{pmatrix},
  \end{align*}
  where, for $\ell=1,2$,
  \begin{align*}
    \begin{aligned}
      &M_{1,\ell} := \int_{\mathbb{R}_{+}^{N}}\sum_{i=1}^{k} U_{\mathbf{x},\Lambda}^{\frac{4}{N-2}} Z_{i,\ell}\left(\frac{\partial W_{r,\Lambda}}{\partial r}+\frac{\partial \phi}{\partial r}\right)+\int_{\partial\mathbb{R}_{+}^{N}}\sum_{i=1}^{k} U_{\mathbf{x},\Lambda}^{\frac{2}{N-2}} Z_{i,\ell}\left(\frac{\partial W_{r,\Lambda}}{\partial r}+\frac{\partial \phi}{\partial r}\right),\\
      &M_{2,\ell} :=  \int_{\mathbb{R}_{+}^{N}}\sum_{i=1}^{k} U_{\mathbf{x},\Lambda}^{\frac{4}{N-2}} Z_{i,\ell}\left(\frac{\partial W_{r,\Lambda}}{\partial \Lambda}+\frac{\partial \phi}{\partial \Lambda}\right)+\int_{\partial\mathbb{R}_{+}^{N}}\sum_{i=1}^{k} U_{\mathbf{x},\Lambda}^{\frac{2}{N-2}} Z_{i,\ell}\left(\frac{\partial W_{r,\Lambda}}{\partial \Lambda}+\frac{\partial \phi}{\partial \Lambda}\right).
    \end{aligned}
  \end{align*}
  From \eqref{eq.Z_i_1_2_expansion_1} and \eqref{eq.Z_i_1_2_expansion_2}, by direct calculation,  we deduce that
  \begin{align*}
    \left|\frac{\partial Z_{i,\ell}}{\partial r}\right| \leq C U_{\mathbf{x}_{i},\Lambda} \quad\text{and}\quad \left|\frac{\partial Z_{i,\ell}}{\partial \Lambda}\right| \leq C U_{\mathbf{x}_{i},\Lambda}, \quad\text{for}~ \ell=1,2.
  \end{align*}
  Then
$$
  \begin{aligned}
&\int_{\mathbb{R}_{+}^{N}}\sum_{i=1}^{k} U_{\mathbf{x},\Lambda}^{\frac{4}{N-2}} Z_{i,1}\left(\frac{\partial W_{r,\Lambda}}{\partial r}+\frac{\partial \phi}{\partial r}\right)=  \int_{\mathbb{R}_{+}^{N}}\sum_{i=1}^{k} U_{\mathbf{x},\Lambda}^{\frac{4}{N-2}} Z_{i,1}\sum_{j=1}^{k}Z_{j,1}+ \int_{\mathbb{R}_{+}^{N}}\sum_{i=1}^{k} U_{\mathbf{x},\Lambda}^{\frac{4}{N-2}} Z_{i,1}\frac{\partial \phi}{\partial r}\\
  &\quad = \int_{\mathbb{R}_{+}^{N}}\sum_{i=1}^{k} U_{\mathbf{x},\Lambda}^{\frac{4}{N-2}} Z_{i,1}Z_{1,1}+\mathcal O\Biggl(\sum_{i=1}^{k}\sum_{j\neq i} \frac{1}{|\mathbf{x}_i-\mathbf{x}_{j}|^{N-2}}\Biggr) + \int_{\mathbb{R}_{+}^{N}}\sum_{i=1}^{k}\phi \frac{\partial }{\partial r}\left(U_{\mathbf{x},\Lambda}^{\frac{4}{N-2}} Z_{i,1}\right)\\
  &\quad = k \int_{\mathbb{R}_{+}^{N}} U_{\mathbf{x},\Lambda}^{\frac{4}{N-2}} Z_{1,1}^2 + k \mathcal O\left(\frac{1}{\mu}\right)^{\mathfrak m} + O\left(\|\phi\|_{*}\int_{\mathbb{R}_{+}^{N}}\sum_{i=1}^{k}\frac{1}{(1+|y-\mathbf{x}_{i}|)^{N}}\sum_{j}^{k} \frac{1}{(1+|y-\mathbf{x}_{j}|)^{\frac{N+2}{2}+\tau}} d y\right)\\
  &\quad =  k \int_{\mathbb{R}_{+}^{N}} U_{\mathbf{x},\Lambda}^{\frac{4}{N-2}} Z_{1,1}^2+ k \mathcal O\left(\frac{1}{\mu}\right)^{\mathfrak m} + kO\left(\frac{1}{\mu}\right)^{\frac{\mathfrak m}{2}+\sigma} \\
  &\quad = k \int_{\mathbb{R}_{+}^{N}} U_{\mathbf{x},\Lambda}^{\frac{4}{N-2}} Z_{1,1}^2 + k\mathcal O\left(\frac{1}{\mu}\right)^{\frac{\mathfrak m}{2}+\sigma}.
  \end{aligned}
$$
Similarly,
\begin{align*}
  \begin{aligned}
    &\int_{\partial\mathbb{R}_{+}^{N}}\sum_{i=1}^{k} U_{\mathbf{x},\Lambda}^{\frac{2}{N-2}} Z_{i,1}\left(\frac{\partial W_{r,\Lambda}}{\partial r}+\frac{\partial \phi}{\partial r}\right)=   k \int_{\partial\mathbb{R}_{+}^{N}} U_{\mathbf{x},\Lambda}^{\frac{2}{N-2}} Z_{1,1}^2 + k\mathcal O\left(\frac{1}{\mu}\right)^{\frac{\mathfrak m}{2}+\sigma}.
\end{aligned}
\end{align*}
Thus, 
\[M_{1,1}= k \int_{\mathbb{R}_{+}^{N}} U_{\mathbf{x},\Lambda}^{\frac{4}{N-2}} Z_{1,1}^2 + k \int_{\partial\mathbb{R}_{+}^{N}} U_{\mathbf{x},\Lambda}^{\frac{2}{N-2}} Z_{1,1}^2 + k\mathcal O\left(\frac{1}{\mu}\right)^{\frac{\mathfrak m}{2}+\sigma}.\]
Using the same discussion, we have
\begin{align*}
  \begin{aligned}
    &M_{2,2}= k \int_{\mathbb{R}_{+}^{N}} U_{\mathbf{x},\Lambda}^{\frac{4}{N-2}} Z_{1,2}^2 + k \int_{\partial\mathbb{R}_{+}^{N}} U_{\mathbf{x},\Lambda}^{\frac{2}{N-2}} Z_{1,2}^2 + k\mathcal O\left(\frac{1}{\mu}\right)^{\frac{\mathfrak m}{2}+\sigma},\\
    &M_{1,2} = k\mathcal O\left(\frac{1}{\mu}\right)^{\frac{\mathfrak m}{2}+\sigma}, \quad  M_{2,1} = k\mathcal O\left(\frac{1}{\mu}\right)^{\frac{\mathfrak m}{2}+\sigma}.
  \end{aligned}
\end{align*}
This completes the proof.
\end{proof}
We can therefore expand the reduced functional. \\ Before stating the result, we remark that,
 for any $\beta_{0} \geq \frac{N-2-\mathfrak m}{N-2}$ and large $k$,
  
  \begin{align*}
    \begin{aligned}
      \sum_{j=2}^k \frac{1}{\left|x_{j}-x_{1}\right|^{\beta_{0}}}  =\frac{1}{2^{\beta_0}} \sum_{j=2}^k \frac{1}{r^{\beta_0}\left(\sin \frac{|j-1|\pi}{k}\right)^{\beta_0}} \leq \frac{C k^{\beta_{0}}}{\mu^{\beta_{0}}} \sum_{j=2}^k \frac{1}{j^{\beta_{0}}}.
    \end{aligned}
  \end{align*}
  
  Since $\mu=k^{\frac{N-2}{N-2-\mathfrak m}}$,  by symmetry, it holds
  \begin{align}\label{eq.lemma.appendix_estimates_3_temp}
    \begin{aligned}
      \sum_{j=2}^k \frac{1}{\left|x_{j}-x_{1}\right|^{\beta_{0}}}  \lesssim
     \begin{cases}\frac{ k^{\beta_{0}} }{\mu^{\beta_{0}}} = O \left(\mu^{-\frac{\mathfrak m\beta_0}{N-2}}\right) , & \beta_{0} > 1, \vspace{10pt}\\ \frac{ k^{\beta_0}\ln k}{\mu^{\beta_{0}}}= O \left(\mu^{-\frac{\mathfrak m}{N-2}}\ln\mu\right), & \beta_{0} = 1, \vspace{10pt}\\ \frac{ k}{\mu^{\beta_{0}}} = O\left(\mu^{\frac{N-2-\mathfrak m}{N-2}-\beta_0}\right), & \beta_{0} < 1.\end{cases}
    \end{aligned}
  \end{align}

\begin{proposition}\label{reduced}
If $k$ is large enough, we have that
$$\begin{aligned} F(r, \Lambda)&=J(W_{r, \Lambda})+ \mathcal O\left(\frac{k}{\mu^{\mathfrak m+\sigma}}\right)\\
&=k\left[A-B\sum_{i=2}^k\frac{1}{\Lambda^{N-2}|\mathbf{x}_i-\mathbf{x}_1|^{N-2}}-\frac{1}{2^*}\frac{c_0}{\Lambda^m \mu^m}\mathfrak d_{3, N}+(N-2)\frac{d_0}{\Lambda^n \mu^n}\mathfrak d_{5, N}\right]\\
&+k\left[-\frac{m(m-1)c_0}{2\cdot2^*}\mathfrak d_{4, N}\frac{(\mu r_0-r)^2}{\Lambda^{m-2}\mu^m}+\frac{n(n-1)(N-2)d_0}{2}\mathfrak d_{6, N}\frac{(\mu r_0-r)^2}{\Lambda^{n-2}\mu^n}\right]\\
&+\mathcal O\left(\frac{k}{\mu^{\mathfrak m+\sigma}}\right)  +\mathcal O\left(\frac{k}{\mu^{\mathfrak m}}|\mu r_0-r|^{2+\theta}\right)
\end{aligned}$$
where $\sigma>0$ and $\theta>0$ are fixed and small and $A,B,\mathfrak d_{i, N}$,  $i=3, 4, 5, 6$, are some positive constants. 
\end{proposition}
\begin{proof}
Arguing as in \cite{WangZhao_2013} and in \cite{WeiYan_2010a} it is easy to show that

$$F(r, \Lambda)=J(W_{r, \Lambda})+\mathcal O\left(\frac{k}{\mu^{\mathfrak m+\sigma}}\right).$$
Now
$$J(W_{r, \Lambda}):=\underbrace{\frac{c_N}{2}\int_{\mathbb R^N_+}|\nabla W_{r, \Lambda}|^2}_{(I)}\underbrace{+\frac{1}{2^*}\int_{\mathbb R^N_+} K\left(\frac{|y|}{\mu}\right)(W_{r, \Lambda})^{2^*}_+-(N-2)\int_{\partial \mathbb R^N_+}H\left(\frac{|\bar y|}{\mu}\right)(W_{r, \Lambda})^{2^\sharp}_+}_{(II)}$$
Now, since $U_{\mathbf{x}_j,\Lambda}$ satisfies \eqref{eq:bubble} we get
$$\begin{aligned}(I)&=\frac{c_N}{2}\sum_{j=1}^k\int_{\mathbb R^N_+}|\nabla U_{\mathbf{x}_j,\Lambda}|^2+\frac{c_N}{2}\sum_{i\neq j}\int_{\mathbb R^N_+}\nabla U_{\mathbf{x}_j,\Lambda}\nabla U_{\mathbf{x}_i,\Lambda}\\
&=-\frac 12 \sum_{j=1}^k \int_{\mathbb R^N_+}U_{\mathbf{x}_j,\Lambda}^{2^*}+\frac{(N-1)}{\sqrt{N(N-1)}}\mathfrak D \sum_{j=1}^k\int_{\partial\mathbb R^N_+}U_{\mathbf{x}_j,\Lambda}^{2^\sharp}\\
&-\frac 12 \sum_{i\neq j}^k \int_{\mathbb R^N_+}U_{\mathbf{x}_j,\Lambda}^{2^*-1}U_{\mathbf{x}_i,\Lambda}+\frac{(N-1)}{\sqrt{N(N-1)}}\mathfrak D \sum_{i\neq j}^k\int_{\partial\mathbb R^N_+}U_{\mathbf{x}_j,\Lambda}^{2^\sharp-1}U_{\mathbf{x}_i,\Lambda}\\
&=-\frac 12 k \mathfrak a_N+\frac{(N-1)}{\sqrt{N(N-1)}}\mathfrak Dk\mathfrak b_N\\
&-\frac 12 k\sum_{i=2}^k \int_{\mathbb R^N_+}U_{\mathbf{x}_1,\Lambda}^{2^*-1}U_{\mathbf{x}_i,\Lambda}+\frac{(N-1)}{\sqrt{N(N-1)}}\mathfrak Dk \sum_{i=2}^k\int_{\partial\mathbb R^N_+}U_{\mathbf{x}_1,\Lambda}^{2^\sharp-1}U_{\mathbf{x}_i,\Lambda}\\
&=-\frac 12 k \mathfrak a_N+\frac{(N-1)}{\sqrt{N(N-1)}}\mathfrak Dk\mathfrak b_N\\
&-\frac 12 k \mathfrak d_{1, N}\sum_{i=2}^k\frac{1}{\Lambda^{N-2}|\mathbf{x}_i-\mathbf{x}_1|^{N-2}}+\frac{(N-1)}{\sqrt{N(N-1)}}\mathfrak Dk\mathfrak d_{2, N}\sum_{i=2}^k\frac{1}{\Lambda^{N-2}|\mathbf{x}_i-\mathbf{x}_1|^{N-2}}\\
&+\mathcal O\left(k\sum_{i=2}^k\frac{1}{|\mathbf{x}_i-\mathbf{x}_1|^{N-2+\sigma}}\right)\\
&=-\frac 12 k \mathfrak a_N+\frac{(N-1)}{\sqrt{N(N-1)}}\mathfrak Dk\mathfrak b_N\\
&+k\left(-\frac 12\mathfrak d_{1, N}+\frac{(N-1)}{\sqrt{N(N-1)}}\mathfrak D\mathfrak d_{2, N} \right)\sum_{i=2}^k\frac{1}{\Lambda^{N-2}|\mathbf{x}_i-\mathbf{x}_1|^{N-2}}+\mathcal O\left(k\sum_{i=2}^k\frac{1}{|\mathbf{x}_i-\mathbf{x}_1|^{N-2+\sigma}}\right)\\
\end{aligned}$$ where
$$\mathfrak a_N:=\int_{\mathbb R^N_+}U_{\mathbf{0},1}^{2^*};\quad \mathfrak b_N:=\int_{\partial\mathbb R^N_+}U_{\mathbf{0},1}^{2^\sharp}$$
and
$$\mathfrak d_{1, N}:=\alpha_N\int_{\mathbb R^N_+}U_{\mathbf{0},1}^{2^*-1};\quad \mathfrak d_{2, N}:=\alpha_N\int_{\partial\mathbb R^N_+}U_{\mathbf{0},1}^{2^\sharp-1}.$$
Here we use also the fact that (by using the explicit expression of $U_{\mathbf{x}_i,\Lambda}$) $$\begin{aligned} \int_{\mathbb R^N_+}U_{\mathbf{x}_1,\Lambda}^{2^*-1}U_{\mathbf{x}_i,\Lambda}&=\alpha_N^{2^*}\int_{\mathbb R^N_+} \frac{\Lambda^{\frac{N+2}{2}}}{\left(\Lambda^2|\bar y-\mathbf{x}_1|^2+(\Lambda y_N+\mathfrak D)^2-1\right)^{\frac{N+2}{2}}}\frac{\Lambda^{\frac{N-2}{2}}}{\left(\Lambda^2|\bar y-\mathbf{x}_i|^2+(\Lambda y_N+\mathfrak D)^2-1\right)^{\frac{N-2}{2}}}\\
&=\alpha_N^{2^*}\int_{\mathbb R^N_+} \frac{1}{\left(|\bar y|^2+(y_N+\mathfrak D)^2-1\right)^{\frac{N+2}{2}}}\frac{1}{\left(|\bar y+\Lambda(\mathbf{x}_1-\mathbf{x}_i)|^2+(y_N+\mathfrak D)^2-1\right)^{\frac{N-2}{2}}}\\
&=\frac{\alpha_N}{\Lambda^{N-2}|\mathbf{x}_1-\mathbf{x}_i|^{N-2}}\int_{\mathbb R^N_+}U_{\mathbf{0},1}^{2^*}=\mathfrak d_{1, N}\frac{1}{\Lambda^{N-2}|\mathbf{x}_1-\mathbf{x}_i|^{N-2}}\end{aligned}$$ and we can argue similarly for $\int_{\partial\mathbb R^N_+}U_{\mathbf{x}_1,\Lambda}^{2^\sharp-1}U_{\mathbf{x}_i,\Lambda}$.\\
Now, by using the estimates (for $a, b>0$ and $p>2$) $$(a+b)^p-a^p-pa^{p-1}b=\mathcal O(a^{p-1}b^2)+\mathcal O(b^p)$$ we get
$$\begin{aligned}(II)&=\frac{1}{2^*}k\int_{\Omega_1}K\left(\frac{|y|}{\mu}\right)U_{\mathbf{x}_1,\Lambda}^{2^*}+k\sum_{j=2}^k\int_{\Omega_1}K\left(\frac{|y|}{\mu}\right)U_{\mathbf{x}_1,\Lambda}^{2^*-1}U_{\mathbf{x}_j,\Lambda}\\
&-(N-2)k\int_{\widetilde\Omega_1}H\left(\frac{|\bar y|}{\mu}\right)U_{\mathbf{x}_1,\Lambda}^{2^\sharp}-2(N-1)k\sum_{j=2}^k\int_{\widetilde\Omega_1}H\left(\frac{|\bar y|}{\mu}\right)U_{\mathbf{x}_1,\Lambda}^{2^\sharp-1}U_{\mathbf{x}_j,\Lambda}\\
&+\mathcal O\left(k\int_{\Omega_1}U_{\mathbf{x}_1,\Lambda}^{2^*-2}\left(\sum_{j=2}^kU_{\mathbf{x}_j,\Lambda}\right)^2\right)+\mathcal O\left(k\int_{\Omega_1}\left(\sum_{j=2}^kU_{\mathbf{x}_j,\Lambda}\right)^{2^*}\right)\\
&+\mathcal O\left(k\int_{\widetilde\Omega_1}U_{\mathbf{x}_1,\Lambda}^{2^\sharp-2}\left(\sum_{j=2}^kU_{\mathbf{x}_j,\Lambda}\right)^2\right)+\mathcal O\left(k\int_{\widetilde\Omega_1}\left(\sum_{j=2}^kU_{\mathbf{x}_j,\Lambda}\right)^{2^\sharp}\right).
\end{aligned}$$
We remark that, by Lemma \ref{lemma.appendix_estimates_1},
for any $y\in\Omega_1,0<\alpha<N-2$ we have, independently on $\Lambda$ (we recall that $\Lambda$ is bounded so it does not influence the rate),
$$\begin{aligned}\sum_{j=2}^kU_{\mathbf{x}_j,\Lambda}&\lesssim \sum_{j=2}^k\frac{1}{(|y-\mathbf{x}_j|+1)^{N-2}}\lesssim \sum_{j=2}^k\frac{1}{(|y-\mathbf{x}_j|+1)^{N-2-\alpha}}\frac{1}{(|y-\mathbf{x}_j|+1)^{\alpha}}\\
&\lesssim \sum_{j=2}^k\frac{1}{|\mathbf{x}_j-\mathbf{x}_1|^{\alpha}}\frac{1}{(|y-\mathbf{x}_1|+1)^{N-2-\alpha}}\lesssim \left(\frac{k}{\mu}\right)^{\alpha}\frac{1}{(|y-\mathbf{x}_1|+1)^{N-2-\alpha}}.\end{aligned}$$
Then, for $\alpha>1$, using \eqref{eq.lemma.appendix_estimates_3_temp} we get
$$\int_{\Omega_1}U_{\mathbf{x}_1,\Lambda}^{2^*-2}\left(\sum_{j=2}^kU_{\mathbf{x}_j,\Lambda}\right)^2\lesssim \left(\frac{k}{\mu}\right)^{2\alpha}\int_{\mathbb R^N_+}\frac{1}{(|y-\mathbf{x}_1|+1)^{2N-2\alpha}}=\mathcal O\left(\left(\frac{k}{\mu}\right)^{2\alpha}\right).$$ 
Choosing $$\alpha:=(\mathfrak m+\sigma)\frac{N-2}{2\mathfrak m},$$
then
$$\int_{\Omega_1}U_{\mathbf{x}_1,\Lambda}^{2^*-2}\left(\sum_{j=2}^kU_{\mathbf{x}_j,\Lambda}\right)^2=\left(\frac{k}{\mu}\right)^{2\alpha}=\mathcal O\left(\frac{1}{\mu^{\mathfrak m+\sigma}}\right).$$
Similarly,
$$\int_{\Omega_1}\left(\sum_{j=2}^kU_{\mathbf{x}_j,\Lambda}\right)^{2^*}\lesssim \left(\frac{k}{\mu}\right)^{2^*\alpha}\int_{\mathbb R^N_+}\frac{1}{(|y-\mathbf{x}_1|+1)^{2^*(N-2-\alpha)}}=\mathcal O\left(\left(\frac{k}{\mu}\right)^{2^*\alpha}\right)$$
provided $1<\alpha<\frac{N-2}{2}$. Then we take
 $$\alpha:=(\mathfrak m+\sigma)\frac{N-2}{2^*\mathfrak m}=\frac{\mathfrak m+\sigma}{\mathfrak m}\frac{(N-2)^2}{2N},$$ so that
$$\int_{\Omega_1}\left(\sum_{j=2}^kU_{\mathbf{x}_j,\Lambda}\right)^{2^*}=\mathcal O\left(\left(\frac{k}{\mu}\right)^{2^*\alpha}\right)=\mathcal O\left(\frac{1}{\mu^{\mathfrak m+\sigma}}\right).$$
With a similar argument, we can also show that
$$\int_{\widetilde\Omega_1}U_{\mathbf{x}_1,\Lambda}^{2^\sharp-2}\left(\sum_{j=2}^kU_{\mathbf{x}_j,\Lambda}\right)^2=\mathcal O\left(\frac{1}{\mu^{\mathfrak m+\sigma}}\right),\quad \int_{\widetilde\Omega_1}\left(\sum_{j=2}^kU_{\mathbf{x}_j,\Lambda}\right)^{2^\sharp}=\mathcal O\left(\frac{1}{\mu^{\mathfrak m+\sigma}}\right).$$
Next, we obtain
$$\begin{aligned}
\sum_{j=2}^k\int_{\Omega_1}K\left(\frac{|y|}{\mu}\right)U_{\mathbf{x}_1,\Lambda}^{2^*-1}U_{\mathbf{x}_j,\Lambda}&=\sum_{j=2}^k\int_{\Omega_1}U_{\mathbf{x}_1,\Lambda}^{2^*-1}U_{\mathbf{x}_j,\Lambda}+\sum_{j=2}^k\int_{\Omega_1}\left(K\left(\frac{|y|}{\mu}\right)-1\right)U_{\mathbf{x}_1,\Lambda}^{2^*-1}U_{\mathbf{x}_j,\Lambda}&\\
&=\mathfrak d_{1, N}\sum_{i=2}^k\frac{1}{\Lambda^{N-2}|\mathbf{x}_i-\mathbf{x}_1|^{N-2}}+\mathcal O\left(\frac{1}{\mu^{\mathfrak m+\sigma}}\right)\end{aligned}$$ since by assumption $(\bf K)$
$$\begin{aligned}\sum_{j=2}^k\int_{\Omega_1}\left(K\left(\frac{|y|}{\mu}\right)-1\right)U_{\mathbf{x}_1,\Lambda}^{2^*-1}U_{\mathbf{x}_j,\Lambda}&\lesssim \frac{1}{\mu^m}\sum_{j=2}^k\frac{1}{|\mathbf{x}_j-\mathbf{x}_1|^{N-2}}\int_{\mathbb R^N_+}||y-\mathbf{x}_1|-\mu r_0|^m U_{\mathbf{0},1}^{2^*-1}\\
&\lesssim \frac{1}{\mu^{m+\sigma}}.\end{aligned}$$ Analogously, it follows that
$$\begin{aligned}\sum_{j=2}^k\int_{\widetilde\Omega_1}H\left(\frac{|\bar y|}{\mu}\right)U_{\mathbf{x}_1,\Lambda}^{2^\sharp-1}U_{\mathbf{x}_j,\Lambda}=&\frac{\mathfrak D}{\sqrt{N(N-1)}}\sum_{j=2}^k\int_{\widetilde\Omega_1}U_{\mathbf{x}_1,\Lambda}^{2^\sharp-1}U_{\mathbf{x}_j,\Lambda}\\&+\sum_{j=2}^k\int_{\widetilde\Omega_1}\left(H\left(\frac{|\bar y|}{\mu}\right)-\frac{\mathfrak D}{\sqrt{N(N-1)}}\right)U_{\mathbf{x}_1,\Lambda}^{2^\sharp-1}U_{\mathbf{x}_j,\Lambda}\\
=&\frac{\mathfrak D}{\sqrt{N(N-1)}}\mathfrak d_{2, N}\sum_{i=2}^k\frac{1}{\Lambda^{N-2}|\mathbf{x}_i-\mathbf{x}_1|^{N-2}}+\mathcal O\left(\frac{1}{\mu^{\mathfrak m+\sigma}}\right)\\
\end{aligned}$$

Now, arguing as in \cite{WeiYan_2010a} we get that
$$
  \begin{aligned}
\int_{\Omega_1}K\left(\frac{|y|}{\mu}\right)U_{\mathbf{x}_1,\Lambda}^{2^*}=& \int_{\mathbb R^N_+} U_{\mathbf{0},1}^{2^*}-\frac{c_0}{\mu^m}\int_{\mathbb R^N_+}||y-\mathbf{x}_1|-\mu r_0|^mU_{\mathbf{0},1}^{2^*}+\mathcal O\left(\frac{1}{\mu^{\mathfrak m+\sigma}}\right)\\
=&\mathfrak a_N-\frac{c_0}{\Lambda^m\mu^m}\int_{\mathbb R^N_+}|y_1|^mU_{\mathbf{0},1}^{2^*}\\&-\frac 12 m(m-1)\frac{c_0}{\Lambda^{m-2}\mu^m}(\mu r_0-r)^2 \int_{\mathbb R^N_+}|y_1|^{m-2}U_{\mathbf{0},1}^{2^*}+\mathcal O\left(\frac{1}{\mu^{\mathfrak m+\sigma}}\right)\\
=&\mathfrak a_N-\frac{c_0}{\Lambda^m\mu^m}\mathfrak d_{3, N}-\frac 12 m(m-1)\frac{c_0}{\Lambda^{m-2}\mu^m} (\mu r_0-r)^2\mathfrak d_{4, N}\\
&+\mathcal O\left(\frac{1}{\mu^{\mathfrak m+\sigma}}\right) +\mathcal O\left(\frac{1}{\mu^{m}}|\mu r_0-r|^{2+\theta}\right)  \end{aligned}
$$
where
$$\mathfrak d_{3, N}:=\int_{\mathbb R^N_+}|y_1|^mU_{\mathbf{0},1}^{2^*},\quad \mathfrak d_{4, N}:= \int_{\mathbb R^N_+}|y_1|^{m-2}U_{\mathbf{0},1}^{2^*}.$$

Finally, arguing as in \cite{WangZhao_2013} we get that

$$\begin{aligned}
\int_{\widetilde\Omega_1}H\left(\frac{|\bar y|}{\mu}\right)U_{\mathbf{x}_1,\Lambda}^{2^\sharp}=&\frac{\mathfrak D}{\sqrt{N(N-1)}}\mathfrak b_N-\frac{d_0}{\Lambda^n\mu^n}\mathfrak d_{5, N}-\frac 12 n(n-1) \frac{d_0}{\Lambda^{n-2}\mu^n}(\mu r_0-r)^2\mathfrak d_{6, N}\\
&+\mathcal O\left(\frac{1}{\mu^{\mathfrak m+\sigma}}\right) +\mathcal O\left(\frac{1}{\mu^{n}}|\mu r_0-r|^{2+\theta}\right)   \end{aligned}$$

where
$$\mathfrak d_{5, N}:=\int_{\partial\mathbb R^N_+}|y_1|^nU_{\mathbf{0},1}^{2^\sharp},\quad \mathfrak d_{6, N}:= \int_{\partial\mathbb R^N_+}|y_1|^{n-2}U_{\mathbf{0},1}^{2^\sharp}.$$

Putting together the various terms we get
$$\begin{aligned}
J(W_{r, \Lambda})=&k\left[A-B\sum_{i=2}^k\frac{1}{\Lambda^{N-2}|\mathbf{x}_i-\mathbf{x}_1|^{N-2}}-\frac{1}{2^*}\frac{c_0}{\Lambda^m \mu^m}\mathfrak d_{3, N}+(N-2)\frac{d_0}{\Lambda^n \mu^n}\mathfrak d_{5, N}\right]\\
&+k\left[-\frac{m(m-1)c_0}{2\cdot2^*}\mathfrak d_{4, N}\frac{(\mu r_0-r)^2}{\Lambda^{m-2}\mu^m}+\frac{n(n-1)(N-2)d_0}{2}\mathfrak d_{6, N}\frac{(\mu r_0-r)^2}{\Lambda^{n-2}\mu^n}\right]\\
&+\mathcal O\left(\frac{k}{\mu^{\mathfrak m+\sigma}}\right)  +\mathcal O\left(\frac{k}{\mu^{\mathfrak m}}|\mu r_0-r|^{2+\theta}\right)\end{aligned}$$
where
$$A:=\left(\frac{1}{2^*}-\frac 1 2\right)\mathfrak a_N +\frac{\mathfrak D}{\sqrt{N(N-1)}}\mathfrak b_N$$
\begin{equation}\label{B}B:=-\frac 1 2\mathfrak d_{1, N}+\frac{N-1}{\sqrt{N(N-1)}}\mathfrak D \mathfrak d_{2, N}.\end{equation}
We remark that $B>0$ (see Remark \ref{oss}).
\end{proof}
Arguing as in Proposition 3.2 of \cite{WeiYan_2010a}
 and Proposition A.2 of  \cite{WangZhao_2013}, we can also show the following result.
 \begin{proposition}\label{der}
If $k$ is large enough, we have that
$$\begin{aligned} \frac{\partial F(r, \Lambda)}{\partial\Lambda}=&k\left[B\sum_{i=2}^k\frac{N-2}{\Lambda^{N-1}|\mathbf{x}_i-\mathbf{x}_1|^{N-2}}+\frac{1}{2^*}\frac{c_0m}{\Lambda^{m+1} \mu^m}\mathfrak d_{3, N}-(N-2)\frac{d_0n}{\Lambda^{n+1} \mu^n}\mathfrak d_{5, N}\right]\\
&+\mathcal O\left(\frac{k}{\mu^{\mathfrak m+\sigma}}\right)  +\mathcal O\left(\frac{k}{\mu^{\mathfrak m}}|\mu r_0-r|^{2}\right).
\end{aligned}$$
 \end{proposition}
 \section{Proof of Theorem \ref{principale}}
 Since
\begin{align*}
  \left|\mathbf{x}_{j}-\mathbf{x}_{1}\right|=2\left|\mathbf{x}_{1}\right| \sin \frac{(j-1) \pi}{k}, \quad j=1, \ldots, k,
\end{align*}
we have
\begin{align*}
\begin{aligned}
& \sum_{j=2}^{k} \frac{1}{\left|\mathbf{x}_{j}-\mathbf{x}_{1}\right|^{N-2}}=\frac{1}{\left(2\left|\mathbf{x}_{1}\right|\right)^{N-2}} \sum_{j=2}^{k} \frac{1}{\left(\sin \frac{(j-1) \pi}{k}\right)^{N-2}} \\
& \quad= \begin{cases}\frac{1}{\left(2\left|\mathbf{x}_{1}\right|\right)^{N-2}} \sum_{j=2}^{\frac{k}{2}} \frac{1}{\left(\sin \frac{(j-1) \pi}{k}\right)^{N-2}}+\frac{1}{\left(2\left|\mathbf{x}_{1}\right|\right)^{N-2}}, & \text{if}~ k ~\text{is even}, \\
\frac{1}{\left(2\left|\mathbf{x}_{1}\right|\right)^{n-2}} \sum_{j=2}^{ \lfloor \frac{k}{2}\rfloor} \frac{1}{\left(\sin \frac{(j-1) \pi}{k}\right)^{N-2}}, & \text {if}~ k ~\text{is odd},\end{cases}
\end{aligned}
\end{align*}
and
\begin{align*}
  0<c' \leq \frac{\sin \frac{(j-1) \pi}{k}}{\frac{(j-1) \pi}{k}} \leq c'', \quad j=2, \ldots,\left\lfloor\frac{k}{2}\right\rfloor,
\end{align*}
where
$$\lfloor x\rfloor:=\max\{n\in\mathbb N:\,\,n\le x\}$$
and $c',c''$ are some positive constants.

So, there exists a constant $B_0>0$, such that
\begin{align*}
\sum_{j=2}^{k} \frac{1}{\left|\mathbf{x}_{j}-\mathbf{x}_{1}\right|^{N-2}}=\frac{B_0 k^{N-2}}{\left|\mathbf{x}_{1}\right|^{N-2}}+O\left(\frac{k}{\left|\mathbf{x}_{1}\right|^{N-2}}\right).
\end{align*}

Thus, by Propositions \ref{reduced} and \ref{der}, we have
$$\begin{aligned} F(r, \Lambda)=&k\left[A-B_1\frac{k^{N-2}}{\Lambda^{N-2}r^{N-2}}-\frac{1}{2^*}\frac{c_0}{\Lambda^m \mu^m}\mathfrak d_{3, N}+(N-2)\frac{d_0}{\Lambda^n \mu^n}\mathfrak d_{5, N}\right]\\
&+k\left[-\frac{m(m-1)c_0}{2\cdot2^*}\mathfrak d_{4, N}\frac{(\mu r_0-r)^2}{\Lambda^{m-2}\mu^m}+\frac{n(n-1)(N-2)d_0}{2}\mathfrak d_{6, N}\frac{(\mu r_0-r)^2}{\Lambda^{n-2}\mu^n}\right]\\
&+\mathcal O\left(\frac{k}{\mu^{\mathfrak m+\sigma}}\right)  +\mathcal O\left(\frac{k}{\mu^{\mathfrak m}}|\mu r_0-r|^{2+\theta}\right)+\mathcal O\left(\frac{k^2}{r^{N-2}}\right)
\end{aligned}$$
where $B_1=B\cdot B_0>0$
and
$$\begin{aligned} \frac{\partial F(r, \Lambda)}{\partial\Lambda}=&k\left[B_1\frac{(N-2)k^{N-2}}{\Lambda^{N-1}r^{N-2}}+\frac{1}{2^*}\frac{c_0m}{\Lambda^{m+1} \mu^m}\mathfrak d_{3, N}-(N-2)\frac{d_0n}{\Lambda^{n+1} \mu^n}\mathfrak d_{5, N}\right]\\
&+\mathcal O\left(\frac{k}{\mu^{\mathfrak m+\sigma}}\right)  +\mathcal O\left(\frac{k}{\mu^{\mathfrak m}}|\mu r_0-r|^{2}\right)+\mathcal O\left(\frac{k^2}{r^{N-2}}\right).
\end{aligned}$$
Now, if $\mathfrak m:=m<n$ then we have that the reduced functional and its derivative become
$$\begin{aligned} F_m(r, \Lambda)=&k\left[A-B_1\frac{k^{N-2}}{\Lambda^{N-2}r^{N-2}}-\frac{1}{2^*}\frac{c_0}{\Lambda^m \mu^m}\mathfrak d_{3, N}-\frac{m(m-1)c_0}{2\cdot2^*}\mathfrak d_{4, N}\frac{(\mu r_0-r)^2}{\Lambda^{m-2}\mu^m}\right]\\
&+\mathcal O\left(\frac{k}{\mu^{m+\sigma}}\right)  +\mathcal O\left(\frac{k}{\mu^{m}}|\mu r_0-r|^{2+\theta}\right)+\mathcal O\left(\frac{k^2}{r^{N-2}}\right)
\end{aligned}$$

and
$$\begin{aligned} \frac{\partial F_m(r, \Lambda)}{\partial\Lambda}=&k\left[B_1\frac{(N-2)k^{N-2}}{\Lambda^{N-1}r^{N-2}}+\frac{1}{2^*}\frac{c_0m}{\Lambda^{m+1} \mu^m}\mathfrak d_{3, N}\right]\\
&+\mathcal O\left(\frac{k}{\mu^{m+\sigma}}\right)  +\mathcal O\left(\frac{k}{\mu^{m}}|\mu r_0-r|^{2}\right)+\mathcal O\left(\frac{k^2}{r^{N-2}}\right).
\end{aligned}$$
Assume $c_0<0$ and  let $\Lambda_{0, m}$ be the solution of
\[
B_1\frac{(N-2)}{\Lambda^{N-1}r^{N-2}_0}+\frac{1}{2^*}\frac{c_0m}{\Lambda^{m+1} }\mathfrak d_{3, N}=0,
\]
that is
\[
\Lambda_{0, m}=\left(\frac{2^*B_1(N-2)}{-c_0m \mathfrak d_{3, N}r_{0}^{N-2}}\right)^{\frac{1}{N-2-m}}.
\]

Now, if  $\mathfrak m=n<m$ then we have that the reduced functional and its derivative become
$$\begin{aligned} F_n(r, \Lambda)=&k\left[A-B_1\frac{k^{N-2}}{\Lambda^{N-2}r^{N-2}}+(N-2)\frac{d_0}{\Lambda^n \mu^n}\mathfrak d_{5, N}+\frac{n(n-1)(N-2)d_0}{2}\mathfrak d_{6, N}\frac{(\mu r_0-r)^2}{\Lambda^{n-2}\mu^n}\right]\\
&+\mathcal O\left(\frac{k}{\mu^{n+\sigma}}\right)  +\mathcal O\left(\frac{k}{\mu^{n}}|\mu r_0-r|^{2+\theta}\right)+\mathcal O\left(\frac{k^2}{r^{N-2}}\right)
\end{aligned}$$

and
$$\begin{aligned} \frac{\partial F_n(r, \Lambda)}{\partial\Lambda}=&k\left[B_1\frac{(N-2)k^{N-2}}{\Lambda^{N-1}r^{N-2}}-(N-2)\frac{d_0 n}{\Lambda^{n+1} \mu^n}\mathfrak d_{5, N}\right]\\
&+\mathcal O\left(\frac{k}{\mu^{n+\sigma}}\right)  +\mathcal O\left(\frac{k}{\mu^{n}}|\mu r_0-r|^{2}\right)+\mathcal O\left(\frac{k^2}{r^{N-2}}\right).
\end{aligned}$$
Assume $d_0>0$ and  let $\Lambda_{0, n}$ be the solution of
\[
B_1\frac{(N-2)}{\Lambda^{N-1}r^{N-2}_0}-(N-2)\frac{d_0 n}{\Lambda^{n+1} }\mathfrak d_{5, N}=0,
\]
that is
\[
\Lambda_{0, n}=\left(\frac{B_1}{d_0n \mathfrak d_{5, N}r_{0}^{N-2}}\right)^{\frac{1}{N-2-n}}.
\]

Finally, if $\mathfrak m=m=n$ then the reduced functional is
$$\begin{aligned} F_{\mathfrak m}(r, \Lambda)=&k\left[A-B_1\frac{k^{N-2}}{\Lambda^{N-2}r^{N-2}}+B_2\frac{1}{\Lambda^{\mathfrak m} \mu^{\mathfrak m}}+ B_3\frac{(\mu r_0-r)^2}{\Lambda^{\mathfrak m-2}\mu^{\mathfrak m}}\right]\\
&+\mathcal O\left(\frac{k}{\mu^{\mathfrak m+\sigma}}\right)  +\mathcal O\left(\frac{k}{\mu^{\mathfrak m}}|\mu r_0-r|^{2+\theta}\right)+\mathcal O\left(\frac{k^2}{r^{N-2}}\right)
\end{aligned}$$
where 
\begin{equation}\label{B2}
B_2:=-\frac{c_0}{2^*}\mathfrak d_{3, N}+d_0(N-2)\mathfrak d_{5, N}
\end{equation}
and
\begin{equation}\label{B3}
B_3:=-\frac{\mathfrak m(\mathfrak m-1)c_0}{2\cdot 2^*}\mathfrak d_{4, N}+\frac{\mathfrak m(\mathfrak m-1)d_0(N-2)}{2}\mathfrak d_{6, N}.
\end{equation}
and
$$\begin{aligned} \frac{\partial F_{\mathfrak m}(r, \Lambda)}{\partial\Lambda}=&k\left[B_1\frac{(N-2)k^{N-2}}{\Lambda^{N-1}r^{N-2}}-B_2\frac{\mathfrak m}{\Lambda^{\mathfrak m+1} \mu^{\mathfrak m}}\right]\\
&+\mathcal O\left(\frac{k}{\mu^{\mathfrak m+\sigma}}\right)  +\mathcal O\left(\frac{k}{\mu^{\mathfrak m}}|\mu r_0-r|^{2}\right)+\mathcal O\left(\frac{k^2}{r^{N-2}}\right).
\end{aligned}$$
Since $B_2>0$ (see Remark \ref{ossprimo}), let $\Lambda_{0, \mathfrak m}$ be the solution of
\[
B_1\frac{(N-2)}{\Lambda^{N-1}r^{N-2}_0}-\frac{B_2\mathfrak m}{\Lambda^{\mathfrak m+1} }=0,
\]
that is
\[
\Lambda_{0, \mathfrak m}=\left(\frac{B_1(N-2)}{B_2 \mathfrak m r_{0}^{N-2}}\right)^{\frac{1}{N-2-\mathfrak m}}.
\]
In any case the structure of $F_j(r, \Lambda)$ with $j=m$ or $j=n$ or $j=\mathfrak m$ is the same. So we let
$$\begin{aligned} F_{j}(r, \Lambda)&=k\left[A-B_1\frac{k^{N-2}}{\Lambda^{N-2}r^{N-2}}+\frac{A_1}{\Lambda^{j} \mu^{j}}+ A_2\frac{(\mu r_0-r)^2}{\Lambda^{j-2}\mu^{j}}\right]\\
&+\mathcal O\left(\frac{k}{\mu^{j+\sigma}}\right)  +\mathcal O\left(\frac{k}{\mu^{j}}|\mu r_0-r|^{2+\theta}\right)+\mathcal O\left(\frac{k^2}{r^{N-2}}\right)
\end{aligned}$$ and
$$\begin{aligned} \frac{\partial F_j(r, \Lambda)}{\partial\Lambda}=&k\left[B_1\frac{(N-2)k^{N-2}}{\Lambda^{N-1}r^{N-2}}-\frac{A_1 j }{\Lambda^{j+1} \mu^j}\right]+\mathcal O\left(\frac{k}{\mu^{j+\sigma}}\right)  +\mathcal O\left(\frac{k}{\mu^{j}}|\mu r_0-r|^{2}\right)+\mathcal O\left(\frac{k^2}{r^{N-2}}\right).
\end{aligned}$$
where 
$$A_1=\left\{\begin{aligned}&-\frac{c_0}{2^*}\mathfrak d_{3, N}\quad &\mbox{if}\,\, j=m\\
&(N-2)d_0\mathfrak d_{5, N}\quad &\mbox{if}\,\, j=n\\
& B_2\quad &\mbox{if}\,\, j=\mathfrak m.\end{aligned}\right.\qquad A_2=\left\{\begin{aligned}&-\frac{m(m-1)c_0}{2\cdot 2^*}\mathfrak d_{4, N}\quad &\mbox{if}\,\, j=m\\
&\frac{n(n-1)(N-2)d_0}{2}\mathfrak d_{6, N}\quad &\mbox{if}\,\, j=n\\
& B_3\quad &\mbox{if}\,\, j=\mathfrak m.\end{aligned}\right.$$
Define
\begin{align*}
  \mathtt D_j=\left\{(r, \Lambda): r \in\Bigl[\mu r_{0}-\frac{1}{\mu^{\bar{\theta}}}, \mu r_{0}+\frac{1}{\mu^{\bar{\theta}}}\Bigr], ~ \Lambda \in\Bigl[\Lambda_{0, j}-\frac{1}{\mu^{\frac{3}{2} \bar{\theta}}}, \Lambda_{0, j}+\frac{1}{\mu^{\frac{3}{2} \bar{\theta}}}\Bigr]\right\},
\end{align*}
where $\bar{\theta}>0$ is a small constant and $j=m$ or $j=n$ or $j=\mathfrak m$ depending on the value of $m$ and $n$.\\

For any $(r, \Lambda) \in \mathtt D_j$, we have
\[
\frac{r}{\mu}=r_{0}+\mathcal O\left(\frac{1}{\mu^{1+\bar\theta}}\right),
\]
which implies $$r^{N-2}=\mu^{N-2}\left(r_0^{N-2}+\mathcal O\left(\frac{1}{\mu^{1+\bar\theta}}\right)\right).$$
Thus, we get
\begin{align}\label{eq.nearly_finished_energy_estimate}
  \begin{aligned}
    F_j(r, \Lambda)=& k\Bigg(A+\left(\frac{A_{1}}{\Lambda^{j}}-\frac{B_{1}}{\Lambda^{N-2} r_{0}^{N-2}}\right) \frac{1}{\mu^{j}}+\frac{A_{2}}{\Lambda^{j-2} \mu^{j}}\left(\mu r_{0}-r\right)^{2} \\
& +O\left(\frac{1}{\mu^{j+\sigma}}+\frac{1}{\mu^{j}}\left|\mu r_{0}-r\right|^{2+\bar\sigma}+\frac{k}{\mu^{N-2}}\right)\Bigg), \quad(r, \Lambda) \in \mathtt D_j ,
  \end{aligned}
\end{align}
and
\begin{align}\label{eq.nearly_finished_partial_energy_estimate}
\begin{aligned}
  \frac{\partial F_j(r, \Lambda)}{\partial \Lambda}= & k\left(-\frac{A_{1} j}{\Lambda^{j+1}}+\frac{B_{1}(N-2)}{\Lambda^{N-1} r_{0}^{N-2}}\right) \frac{1}{\mu^{j}}  +O\left(\frac{k}{\mu^{j+\sigma}}+\frac{k}{\mu^{j}}\left|\mu r_{0}-r\right|^{2}+\frac{k^2}{\mu^{N-2}}\right),\quad(r, \Lambda) \in \mathtt D_j.
\end{aligned}
\end{align}
\begin{proof}[Proof of Theorem \ref{principale}]
The existence of a critical point of $F_j(r, \Lambda)$ in $\mathtt D_j$ can be proved as in Propositions 3.3, 3.4 of \cite{WeiYan_2010a}.\\  Moreover, arguing as in \cite{WangZhao_2013}, the solution is also positive.\end{proof}
\begin{remark}\label{oss}
In this remark we will show that $B$ which is defined in \eqref{B} is positive.\\
First we define $$I_m^\alpha:=\int_0^{+\infty}\frac{\rho^\alpha}{(1+\rho^2)^m}\, d\rho\quad\mbox{for}\,\, \alpha+1<2m$$ and $$\varphi_m:=\int_{\mathfrak D}^{+\infty}\frac{1}{(t^2-1)^m}\, dt.$$
By using Lemma A.1 of \cite{Cruz-BlazquezPistoiaVaira_2022} we get that
$$\begin{aligned}\mathfrak d_{1, N}&:=\alpha_N^{\frac{2N}{N+2}}\int_{\mathbb R^N_+}\frac{1}{\left(|\bar z|^2+(z_N+\mathfrak D)^2-1\right)^{\frac{N+2}{2}}}\, dz=\alpha_N^{\frac{2N}{N+2}}\omega_{N-1}I_{\frac{N+2}{2}}^{N-2}\varphi_{\frac 32}\\
&=\left(4N(N-1)\right)^{\frac N 2}\omega_{N-1}I_{\frac N 2+1}^{N-2}\left(-1+\frac{\mathfrak D}{\sqrt{\mathfrak D^2-1}}\right)\end{aligned}$$
in view of the explicit computation of $\varphi_{\frac 32}$.\\ Moreover
$$\begin{aligned}\mathfrak d_{2, N}&:=\alpha_N^{\frac{2(N-1)}{N+2}}\int_{\partial\mathbb R^N_+}\frac{1}{\left(|\bar z|^2+\mathfrak D^2-1\right)^{\frac{N}{2}}}\, d\bar z=\alpha_N^{\frac{2(N-1)}{N+2}}\omega_{N-1}I_{\frac{N}{2}}^{N-2}\frac{1}{\sqrt{\mathfrak D^2-1}}\\
&=\left(4N(N-1)\right)^{\frac{ N-1}{ 2}}\omega_{N-1}N I_{\frac N 2+1}^{N-2}\frac{1}{\sqrt{\mathfrak D^2-1}}\end{aligned}$$
since $$I_{\frac N2}^{N-2}=N I_{\frac N 2+1}^{N-2}.$$ Then we compute $B$ defined in \eqref{B}.
$$\begin{aligned} B&=\left(4N(N-1)\right)^{\frac{ N-1}{ 2}}\omega_{N-1} I_{\frac N 2+1}^{N-2}\left[-\frac12\left(4N(N-1)\right)^{\frac{ 1}{ 2}}\left( -1+\frac{\mathfrak D}{\sqrt{\mathfrak D^2-1}}\right)+(N(N-1))^{\frac 12}\frac{\mathfrak D}{\sqrt{\mathfrak D^2-1}}\right]\\
&=\left(4N(N-1)\right)^{\frac{ N-1}{ 2}}(N(N-1))^{\frac 12}\omega_{N-1} I_{\frac N 2+1}^{N-2}=\frac{\alpha_N^{2^*}}{2}\omega_{N-1} I_{\frac N 2+1}^{N-2}>0.\end{aligned}$$
\end{remark}
\begin{remark}\label{ossprimo}
We remark also that if $c_0<0$ then $B_2>0$ and $B_3>0$. 
\end{remark}
\section{Basic Estimates}\label{appendice}
For each fixed $i$ and $j$, $i \neq j$, we consider the following function

\begin{equation}
g_{i j}(y)=\frac{1}{\left(1+\left|y-\mathbf{x}_{j}\right|\right)^{\alpha}} \frac{1}{\left(1+\left|y-\mathbf{x}_{i}\right|\right)^{\beta}} ,
\end{equation}
where $\alpha \geq 1$ and $\beta \geq 1$ are two constants.
We first introduce the following lemma, whose proof can be found in \cite[Appendix B]{WeiYan_2010a}.

\begin{lemma}\label{lemma.appendix_estimates_1}
  For any $\sigma\in(0,\min\{\alpha,\beta\})$, there is a constant $C>0$, such that
\begin{align}
  g_{i j}(y) \leq \frac{C}{\left|\mathbf{x}_{i}-\mathbf{x}_{j}\right|^{\sigma}}\left(\frac{1}{\left(1+\left|y-\mathbf{x}_{i}\right|\right)^{\alpha+\beta-\sigma}}+\frac{1}{{\left(1+\left|y-\mathbf{x}_{j}\right|\right)}^{\alpha+\beta-\sigma}}\right).
\end{align}
\end{lemma}

\begin{lemma}\label{lemma.appendix_estimates_2}

For any constant $\sigma\in(0,N-2)$, there is a constant $C>0$, such that
\begin{align}
  \int_{\mathbb{R}_{+}^{N}} \frac{1}{|y-z|^{N-2}} \frac{1}{(1+|z|)^{2+\sigma}} d z \leq \frac{C}{{(1+|y|)}^{\sigma}},
\end{align}
and
\begin{align}
  \int_{\partial \mathbb{R}_{+}^{N}} \frac{1}{|y-z|^{N-2}} \frac{1}{(1+|z|)^{1+\sigma}} d \bar{z} \leq \frac{C}{{(1+|y|)}^{\sigma}}.
\end{align}
\end{lemma}
The result is well known. See also \cite[Lemma B.2]{WeiYan_2010a} and \cite[Lemma A.4]{WangZhao_2013}.

\begin{lemma}\label{lemma.appendix_estimates_3}
  Suppose that $N\geq 5$ and $\tau_1,\tau_2 \in \left(0,\frac{3}{2}\right)$. Then there exist small $\theta_1,\theta_2>0$, such that
  \begin{align*}
    \begin{aligned}
      & \int_{\mathbb{R}_{+}^{N}} \frac{1}{|y-z|^{N-2}} W_{r, \Lambda}^{\frac{4}{N-2}}(z) \sum_{j=1}^{k} \frac{1}{\left(1+\left|z-\mathbf{x}_{j}\right|\right)^{\frac{N-2}{2}+\tau_1}} d z \leq C \sum_{j=1}^{k} \frac{1}{\left(1+\left|y-\mathbf{x}_{j}\right|\right)^{\frac{N-2}{2}+\tau_1+\theta_1}},
      \end{aligned}
  \end{align*}
  and   
  \begin{align*}
    \begin{aligned}
      & \int_{\partial\mathbb{R}_{+}^{N}} \frac{1}{|y-z|^{N-2}} W_{r, \Lambda}^{\frac{2}{N-2}}(z) \sum_{j=1}^{k} \frac{1}{\left(1+\left|z-\mathbf{x}_{j}\right|\right)^{\frac{N-2}{2}+\tau_2}} d \bar{z} \leq C \sum_{j=1}^{k} \frac{1}{\left(1+\left|y-\mathbf{x}_{j}\right|\right)^{\frac{N-2}{2}+\tau_2+\theta_2}}.
      \end{aligned}
  \end{align*}

\end{lemma}

\end{document}